\documentclass[11pt,a4paper]{article}

\usepackage[margin=28mm]{geometry}
\usepackage{amsmath,amssymb,amsthm}
\usepackage{graphicx,booktabs,multirow,longtable}
\usepackage{subfigure}
\usepackage[font=small,labelfont=bf]{caption}
\usepackage{microtype}
\usepackage[hidelinks]{hyperref}
\usepackage{url}

\numberwithin{equation}{section}
\newtheorem{theorem}{Theorem}[section]
\newtheorem{lemma}[theorem]{Lemma}

\newtheorem{corollary}[theorem]{Corollary}
\theoremstyle{definition}
\newtheorem{assumption}[theorem]{Assumption}
\newtheorem{definition}[theorem]{Definition}
\newtheorem{example}[theorem]{Example}
\theoremstyle{remark}
\newtheorem{remark}[theorem]{Remark}

\newcommand{\RR}{\mathbb R}

\newcommand{\trt}{^{\mathsf T}}

\hypersetup{
  pdftitle={Ghost Stochastic Approximation for Nonconvex Optimization with Expectation Constraints},
  pdfauthor={Harsh Choudhary, Francisco Facchinei, Vyacheslav Kungurtsev, Lorenzo Lampariello, Gesualdo Scutari},
  pdfsubject={Stochastic approximation, ghost penalties, expectation constraints, randomized multilevel estimation}
}

\title{Ghost Stochastic Approximation for Nonconvex Optimization with Expectation Constraints}
\author{Harsh Choudhary\thanks{Department of Computer Science, Faculty of Electrical Engineering, Czech Technical University in Prague, Prague, Czech Republic.}
  \and Francisco Facchinei\thanks{Department of Computer, Control, and Management Engineering Antonio Ruberti, Sapienza University of Rome, Rome, Italy.}
  \and Vyacheslav Kungurtsev\footnotemark[1]
  \and Lorenzo Lampariello\thanks{Department of Business Studies, Roma Tre University, Rome, Italy.}
  \and Gesualdo Scutari\thanks{Edwardson School of Industrial Engineering, Purdue University, West Lafayette, IN, USA.}}
\date{November 30, 2026}

\begin{document}
\maketitle

\begin{abstract}
We study nonconvex optimization with an expectation objective and expectation
inequality constraints. Ghost stochastic approximation updates the iterate
using an unbiased estimate of the solution of a relaxed, strongly convex
quadratic subproblem. A nonsmooth penalty function is used only in the
analysis. Under bounded iterates, an extended constraint qualification,
Robbins--Monro stepsizes, and conditionally unbiased direction noise with
uniformly bounded second moments, we prove almost-sure asymptotic feasibility
and a vanishing lower limit of the population direction norm. Under the
additional condition that the objective values at feasible stationary points
contain no nondegenerate interval, we prove that the entire direction sequence
converges to zero and every accumulation point satisfies the Karush--Kuhn--Tucker
conditions. The proof controls signed martingale tails before taking the
maximum over constraint violations. We also give a randomized multilevel
direction estimator and precise sufficient conditions for unbiasedness,
finite variance, and finite expected sampling cost. A scalar counterexample
shows why these estimator properties do not follow from regularity of the
outer quadratic subproblem alone. Numerical examples illustrate the method
and distinguish practical sampled variants from the hypotheses of the
convergence theory.
\end{abstract}


\section{Introduction}
\label{s:introduction}

We consider the problem
\begin{equation}\label{eq:prob}
  \min_{x\in\RR^n} F(x)
  \qquad\text{subject to}\qquad C_i(x)\le0,\quad i=1,\ldots,m,
\end{equation}
where the objective and constraint functions are generally nonconvex and
\begin{equation}\label{eq:expectation_problem}
  F(x)=\mathbb E[f(x,\xi)],\qquad
  C_i(x)=\mathbb E[c_i(x,\zeta)].
\end{equation}
The pair $(\xi,\zeta)$ has a fixed joint probability law; its two components
need not be independent. Unless otherwise stated, different sample pairs are
independent. Expectations and derivatives below refer to the population
functions $F$ and $C_i$. We assume that the sampled derivative oracles are
unbiased for these derivatives. In particular, when sample gradients are
used, differentiation under the expectation must be valid. For a
decision-dependent probability law, this property is a separate requirement:
the gradient of the integrand alone need not be an unbiased population
gradient.

Expectation constraints arise when a learning or engineering system must
optimize average performance while meeting average loss, resource, or
reliability requirements. In large finite-sum problems, the same formulation
represents sampling from a dataset. A practical difficulty is that replacing
the coefficients of a constrained subproblem by unbiased sample averages
does not usually produce an unbiased estimate of its optimizer. A second
difficulty is that the maximum constraint violation is nonsmooth. Although
each signed linear noise increment has conditional mean zero, the maximum
of those increments generally does not.

The deterministic ghost-penalty framework of
\cite{facchinei2017ghost,ghostEquality2022} constructs a direction by solving
a strongly convex approximation with a relaxed constraint set. Its penalty
function is an analytical device; no penalty parameter is updated or
evaluated by the algorithm. We retain this structure and study a
Robbins--Monro update driven by an unbiased noisy estimate of the population
direction. The noise variance is allowed to stay positive, and the
convergence proof does not require a deterministically increasing batch size.

Stochastic approximation and sample-average approximation provide
complementary approaches to expectation problems
\cite{kushner2003stochastic,shapiro2009lectures}. Methods for functional
constraints include penalty procedures \cite{pflug1981convergence},
stochastic first-order and proximal methods
\cite{boob2022stochastic}, and augmented Lagrangian methods for nonconvex
expectation constraints \cite{liStocIALM2024}. Stochastic sequential quadratic
optimization has also been studied with deterministic nonlinear constraints
\cite{berahas2021sequential,curtisMixed2024} and with expectation equality
constraints \cite{shenExpectation2025}. The present analysis concerns
expectation inequalities and separates two questions: convergence of the
outer stochastic approximation under a valid direction oracle, and
construction of such an oracle at finite expected sampling cost.

The first contribution is a pathwise convergence proof based on the original
nonsmooth ghost penalty. We accumulate the signed constraint noise before
taking a maximum, obtaining a descent bound on arbitrary iteration
intervals. This yields asymptotic feasibility and
$\liminf_k\|d(x^k)\|=0$. With an explicit weak Sard condition on stationary
objective values, a level-crossing argument gives convergence of the penalty
values; a separate excursion argument then proves the stronger statement
$\limsup_k\|d(x^k)\|=0$. Thus every accumulation point is stationary. This
does not, by itself, assert convergence of the iterates to a single point.

The second contribution is a precise account of randomized multilevel
estimation of the direction, following the debiasing approach of
\cite{blanchet2015unbiased,blanchet2019unbiased}. We state the required
antithetic second-moment rate explicitly and prove a sufficient condition
using fourth moments and a uniform quadratic remainder for the complete
coefficient-to-direction map. This map includes the computation of the
relaxation parameter. A counterexample shows that linear independence and
strict complementarity for the outer quadratic subproblem alone do not
ensure the required rate. The finite-variance, finite-expected-cost
estimator result is therefore conditional on additional regularity; it is
not a consequence of the outer convergence assumptions.

Sections~\ref{s:trust_reg} and~\ref{s:conv} define the population direction and establish
the deterministic and stochastic auxiliary results before their use in the
convergence theorem. Section~\ref{s:estimator} treats the sampled subproblem
and the unbiased estimator. Section~\ref{s:numerics} presents numerical
examples, including practical variants whose relationship to the theory is
stated explicitly.

\paragraph{Notation.}
The unadorned norm is Euclidean, and $\|\cdot\|_\infty$ is the maximum norm.
For a scalar $a$, $a_+=\max\{a,0\}$. The vector $e\in\RR^m$ has all entries
equal to one. All vector inequalities are componentwise. We distinguish
population quantities $F,C,d$ from sampled quantities $f,c,d_N$ and from the
randomized unbiased direction $\widetilde d$. Almost-sure assertions are
with respect to the randomness used throughout the algorithm.

\section{Population directions and regularity}\label{s:trust_reg}
We first define the deterministic direction that the stochastic estimator must
target.  Throughout, $\|\cdot\|$ denotes the Euclidean norm, $e\in\mathbb R^m$
is the vector of ones, and $a_+:=\max\{a,0\}$.  Fix
$0<\lambda<1$ and $0<\rho<\beta$, and let
$B_\beta:=\{d:\|d\|_\infty\le\beta\}$.

\begin{assumption}[Population regularity]\label{as:lip}
The functions $F$ and $C_i$, $i=1,\ldots,m$, are continuously differentiable
on an open convex set $U\subseteq\mathbb R^n$.  Their gradients are Lipschitz
continuous on $U$, with constants $L_F$ and $L_C$, respectively, where one
common constant $L_C$ is used for the finitely many constraints.  The symmetric
matrix-valued function $H$ is continuous on $U$ and satisfies
\begin{equation}\label{eq:psd_H}
    cI\preceq H(x)\preceq\overline c I
    \qquad(x\in U),\qquad 0<c\le\overline c<\infty.
\end{equation}
\end{assumption}

Define the maximum violation and its best linearized value by
\[
\begin{aligned}
\phi(x)&:=\max\{0,C_1(x),\ldots,C_m(x)\},\\
q(x)&:=\min_{\|v\|_\infty\le\rho}
           \max_{1\le i\le m}\bigl(C_i(x)+\nabla C_i(x)^Tv\bigr)_+.
\end{aligned}
\]
The relaxation used in the population subproblem is
\begin{equation}\label{eq:fullkappa}
\kappa(x):=(1-\lambda)\phi(x)+\lambda q(x),
\end{equation}
and the required reduction in the linearized violation is
\begin{equation}\label{eq:theta}
\theta(x):=\phi(x)-\kappa(x)=\lambda\bigl(\phi(x)-q(x)\bigr).
\end{equation}
The population direction is the unique solution of the strongly convex
quadratic program
\begin{equation}\label{eq:subpfull}
\begin{aligned}
d(x):=\arg\min_{d\in B_\beta}\quad
  &\widehat F(x,d):=\nabla F(x)^Td+\tfrac12 d^TH(x)d\\
\text{subject to}\quad
  &\widehat C_i(x,d):=C_i(x)+\nabla C_i(x)^Td\le\kappa(x),
       \qquad i=1,\ldots,m.
\end{aligned}
\end{equation}
The minimum defining $q(x)$ is attained.  Any of its minimizers is feasible
for~\eqref{eq:subpfull}, because $0\le q(x)\le\kappa(x)$ and $\rho<\beta$.
Thus the population subproblem is always feasible, independently of whether
$x$ is feasible for the original problem.  Strong convexity then gives
existence and uniqueness of $d(x)$, and
\begin{equation}\label{eq:direction-bound}
\|d(x)\|\le B_d:=\sqrt n\,\beta.
\end{equation}

\subsection{Constraint qualification and deterministic estimates}
The ordinary MFCQ is a condition at feasible points.  Since the algorithm
also visits infeasible points, we use its extended version from the ghost
penalty framework~\cite{facchinei2017ghost}.

\begin{definition}[Extended MFCQ]\label{def:emfcq}
Let $I(x):=\{i:C_i(x)=\phi(x)\}$.  The extended Mangasarian--Fromovitz
constraint qualification (eMFCQ) holds at $x$ if there is a vector $v$ such that
\[
\nabla C_i(x)^Tv<0\qquad\text{for every }i\in I(x).
\]
The condition is vacuous if $I(x)$ is empty.  At a feasible point it is
precisely the usual MFCQ for the active constraints; at an infeasible point
it requires a common descent direction for the maximally violated constraints.
\end{definition}

All the conclusions below are local to a compact set of possible iterates.
In particular, no constraint qualification is required at unrelated points
of $\mathbb R^n$.

\begin{lemma}[Feasibility reduction and robust strict feasibility]
\label{lem:slater}
Under Assumption~\ref{as:lip}, $q$, $\kappa$, and $\theta$ are continuous on $U$,
and $0\le\theta(x)\le\phi(x)$.  If eMFCQ holds at $x$, then:
\begin{enumerate}
\item $\theta(x)=0$ if $x$ is feasible, whereas $\theta(x)>0$ if $x$ is infeasible;
\item there is a vector $v_x\in\operatorname{int}B_\beta$ satisfying
      $C_i(x)+\nabla C_i(x)^Tv_x<\kappa(x)$ for every $i$.
\end{enumerate}
Consequently, if eMFCQ holds at every point of a compact set $\chi\subset U$,
there is an $s>0$ such that for every $x\in\chi$ one can choose
$v_x\in\operatorname{int}B_\beta$ with
\begin{equation}\label{eq:uniform-slater}
C_i(x)+\nabla C_i(x)^Tv_x\le\kappa(x)-s,
\qquad i=1,\ldots,m.
\end{equation}
\end{lemma}
\begin{proof}
The function minimized in the definition of $q$ is jointly continuous in
$(x,v)$, and its feasible set is a fixed compact box.  Uniform continuity
on a compact neighborhood of any given $x$ therefore implies continuity of
$q$; explicitly, the difference of two minima is bounded by the supremum
of the difference of their objective functions on that box.  The remaining
continuity statements follow from~\eqref{eq:fullkappa}--\eqref{eq:theta}.
The choice $v=0$ shows $0\le q(x)\le\phi(x)$.

If $\phi(x)=0$, then $q(x)=\kappa(x)=\theta(x)=0$.
An MFCQ direction, multiplied by a sufficiently small positive number,
makes all active affine constraints strictly negative while preserving
strict negativity of the inactive constraints, and lies in
$\operatorname{int}B_\beta$.  If there are no active constraints, $v_x=0$
suffices.

If $\phi(x)>0$, an eMFCQ direction strictly decreases all rows attaining
$\phi(x)$.  For a sufficiently small positive multiple of this direction,
the remaining rows and the zero row are still strictly below $\phi(x)$,
and its infinity norm is at most $\rho$.  Hence $q(x)<\phi(x)$, so
$q(x)<\kappa(x)<\phi(x)$.  A minimizer in the definition of $q(x)$ then
provides the required strictly feasible $v_x$, with
$\|v_x\|_\infty\le\rho<\beta$.

Finally, fix one such $v_x$ at each $x\in\chi$.  Continuity preserves its
strict affine inequalities, with at least half their margin, in a
neighborhood of $x$.  A finite subcover of $\chi$ gives a common positive
margin $s$.  Each point can use the fixed vector associated with one of the
neighborhoods containing it.
\end{proof}

\begin{lemma}[Continuity, multipliers, and stationarity]
\label{prop:dtheta}
Suppose Assumption~\ref{as:lip} holds and eMFCQ holds at every point of a
compact set $\chi\subset U$.  Then $d$ is continuous on $\chi$ and the
subproblem~\eqref{eq:subpfull} has nonnegative Lagrange multipliers
$\mu(x)\in\mathbb R^m_+$ satisfying
\begin{equation}\label{eq:kktnutris}
\begin{aligned}
0&\in\nabla F(x)+H(x)d(x)+\nabla C(x)\mu(x)+N_{B_\beta}(d(x)),\\
0&\le\mu_i(x)\ \perp\
  \kappa(x)-C_i(x)-\nabla C_i(x)^Td(x)\ge0.
\end{aligned}
\end{equation}
Here $\nabla C(x)$ has columns $\nabla C_i(x)$ and $N_{B_\beta}$ is the
normal cone of the box.  There is a finite constant $\Lambda$, independent
of $x\in\chi$, such that every such multiplier satisfies
$\|\mu(x)\|_1\le\Lambda$.  In particular,
\begin{equation}\label{eq:convscfinter}
\nabla F(x)^Td(x)\le-c\|d(x)\|^2+\Lambda\theta(x).
\end{equation}
If $\mathcal S$ denotes the feasible KKT set of~\eqref{eq:prob}, then
\begin{equation}\label{eq:zero-direction-kkt}
d(x)=0\quad\Longleftrightarrow\quad x\in\mathcal S,
\qquad x\in\chi.
\end{equation}
\end{lemma}
\begin{proof}
For continuity, let $x_j\to x$ with $x_j,x\in\chi$ and take any convergent
subsequence of $d(x_j)$, with limit $\bar d\in B_\beta$.
Continuity of the constraint data shows that $\bar d$ is feasible at $x$.
Choose a strictly feasible $v_x$ from Lemma~\ref{lem:slater}.  For
$0<t\le1$, the vector $d_t=(1-t)d(x)+t v_x$ is strictly feasible at $x$,
lies in the interior of the box, and remains feasible at $x_j$ for all
sufficiently large $j$.  Optimality at $x_j$ gives
$\widehat F(x_j,d(x_j))\le\widehat F(x_j,d_t)$.
Pass to the subsequence limit and then let $t\downarrow0$ to obtain
$\widehat F(x,\bar d)\le\widehat F(x,d(x))$.  Uniqueness implies
$\bar d=d(x)$.  Since every convergent subsequence has this limit, $d$ is
continuous on $\chi$ and hence uniformly continuous there.

Strict feasibility gives the convex KKT conditions~\eqref{eq:kktnutris}.
For any corresponding multiplier, $d(x)$ minimizes the Lagrangian over
$B_\beta$.  Testing it at the vector in~\eqref{eq:uniform-slater} and using
complementarity gives
\[
\widehat F(x,d(x))
\le \widehat F(x,v_x)-s\|\mu(x)\|_1.
\]
Let $Q:=\max_{x\in\chi,\,v\in B_\beta}|\widehat F(x,v)|<\infty$.
Thus $\|\mu(x)\|_1\le2Q/s$, giving a uniform choice of $\Lambda$.

Let $z\in N_{B_\beta}(d(x))$ be the normal in the stationarity condition.
Since $0\in B_\beta$, $z^Td(x)\ge0$.  Taking the inner product of
stationarity with $d(x)$ and using complementarity yields
\[
\begin{aligned}
\nabla F(x)^Td(x)
&\le-d(x)^TH(x)d(x)
      +\mu(x)^T\bigl(C(x)-\kappa(x)e\bigr)\\
&\le-c\|d(x)\|^2+\|\mu(x)\|_1\theta(x),
\end{aligned}
\]
which proves~\eqref{eq:convscfinter}.

If $d(x)=0$ and $x$ were infeasible, subproblem feasibility would imply
$\phi(x)\le\kappa(x)$, contradicting Lemma~\ref{lem:slater}.
Thus $x$ is feasible and $\kappa(x)=0$.  The box normal vanishes at
$d(x)=0$, so~\eqref{eq:kktnutris} reduces to the KKT conditions of the
original problem.  Conversely, at a feasible KKT point the same
multipliers make $d=0$ satisfy the convex subproblem KKT conditions;
strong convexity makes it the unique minimizer.
\end{proof}

\section{Stochastic approximation and convergence}\label{s:conv}
Let $\mathcal F_\nu$ contain the initial iterate and all randomness revealed
before sampling at iteration $\nu$.  The iterate $x^\nu$ is
$\mathcal F_\nu$-measurable.  The algorithm obtains a random direction
$\widetilde d^\nu$ and updates
\begin{equation}\label{eq:stochiter}
x^{\nu+1}=x^\nu+\gamma^\nu\widetilde d^\nu,
\qquad
\widetilde d^\nu=d(x^\nu)+M_\nu,
\qquad \nu=0,1,\ldots.
\end{equation}
The positive, deterministic stepsizes satisfy
\begin{equation}\label{eq:stepsizereq}
0<\gamma^\nu\le1,\qquad
\sum_{\nu=0}^{\infty}\gamma^\nu=\infty,\qquad
\sum_{\nu=0}^{\infty}(\gamma^\nu)^2<\infty.
\end{equation}
For example, $\gamma^\nu=a(\nu+1)^{-p}$ has these properties if
$0<a\le1$ and $1/2<p\le1$.

\begin{assumption}[Conditional direction estimates]\label{as:dtilde}
The noise $M_\nu$ is $\mathcal F_{\nu+1}$-measurable and, for a finite
constant $\sigma^2$ independent of $\nu$,
\[
\mathbb E[M_\nu\mid\mathcal F_\nu]=0,
\qquad
\mathbb E[\|M_\nu\|^2\mid\mathcal F_\nu]\le\sigma^2
\quad\text{almost surely}.
\]
\end{assumption}
The population direction, rather than a sample objective gradient alone,
must be conditionally unbiased.  The construction and the extra assumptions
needed for such an estimator are addressed in the next section.

\begin{assumption}[Bounded iterates and regularity region]\label{as:bound}
There is a deterministic compact set $\chi\subset U$ such that
$x^\nu\in\chi$ for every $\nu$ almost surely, and eMFCQ holds at every point
of $\chi$.
\end{assumption}
Boundedness is a stability assumption, not a consequence of the box in the
population subproblem: the random estimator $\widetilde d^\nu$ need not lie
in that box.  No projection or truncation is part of~\eqref{eq:stochiter}.
The compactness assumption is explicit so that the theorem does not assert
stability of an otherwise unbounded stochastic recursion.

\subsection{Auxiliary lemmas}\label{s:lemmas}
We collect the probability and interval-descent facts before their use.

\begin{lemma}[Square-summable martingale increments]\label{thm:l2-martingale}
Let $(Y_\nu)$ be scalar martingale differences with respect to
$(\mathcal F_\nu)$: $Y_\nu$ is $\mathcal F_{\nu+1}$-measurable and
$\mathbb E[Y_\nu\mid\mathcal F_\nu]=0$.  If
$\sum_\nu\mathbb E|Y_\nu|^2<\infty$, then
$Z_k:=\sum_{\nu=0}^{k-1}Y_\nu$ converges almost surely and in $L^2$.
On a probability-one event,
\begin{equation}\label{thm:cauchy-tails}
\sup_{\ell\ge n}|Z_\ell-Z_n|\longrightarrow0.
\end{equation}
The same conclusions hold componentwise for a finite-dimensional vector.
\end{lemma}
\begin{proof}
Orthogonality of martingale differences gives
$\mathbb E|Z_\ell-Z_n|^2=\sum_{\nu=n}^{\ell-1}\mathbb E|Y_\nu|^2$,
so $Z_k$ is Cauchy in $L^2$.  For completeness, almost-sure convergence
follows from the $L^2$ martingale maximal inequality
\cite{durrett2019}.  With
$v_n:=\sum_{\nu=n}^{\infty}\mathbb E|Y_\nu|^2$, that inequality, first
on a finite horizon and then by monotone convergence, gives
\[
\mathbb E\!\left[\sup_{\ell\ge n}|Z_\ell-Z_n|^2\right]\le4v_n.
\]
Choose $n_j\uparrow\infty$ so that $v_{n_j}\le2^{-4j}$.
Markov's inequality and the Borel--Cantelli lemma imply that, almost surely,
$\sup_{\ell\ge n_j}|Z_\ell-Z_{n_j}|\le2^{-j}$ for all sufficiently large
$j$.  For $\ell\ge n\ge n_j$, the triangle inequality then bounds
$|Z_\ell-Z_n|$ by $2^{1-j}$.  This proves the Cauchy property and the
uniform tail statement~\eqref{thm:cauchy-tails}.  The almost-sure limit
agrees with the $L^2$ limit.
\end{proof}

\begin{lemma}[Noise and displacement tails]\label{lem:noise-tails}
Under Assumptions~\ref{as:lip}, \ref{as:dtilde}, and~\ref{as:bound}, and
the stepsize conditions~\eqref{eq:stepsizereq}, there is a probability-one
event on which all the following quantities tend to zero:
\begin{align}
b_n&:=\sup_{\ell\ge n}
       \left\|\sum_{\nu=n}^{\ell-1}\gamma^\nu M_\nu\right\|,
\label{eq:raw-noise-tail}\\
r_n&:=\max_{1\le i\le m}\sup_{\ell\ge n}
       \left|\sum_{\nu=n}^{\ell-1}
             \gamma^\nu\nabla C_i(x^\nu)^TM_\nu\right|,
\label{eq:constraint-noise-tail}\\
s_n&:=\sup_{\ell\ge n}
       \left|\sum_{\nu=n}^{\ell-1}
             \gamma^\nu\nabla F(x^\nu)^TM_\nu\right|,
\label{eq:objective-noise-tail}\\
Q_n&:=\sum_{\nu=n}^{\infty}\|x^{\nu+1}-x^\nu\|^2.
\label{eq:quadratic-tail}
\end{align}
In particular, $x^{\nu+1}-x^\nu\to0$ almost surely, and, with
$T_{n,\ell}:=\sum_{\nu=n}^{\ell-1}\gamma^\nu$,
\begin{equation}\label{eq:displacement-window}
\|x^\ell-x^n\|\le B_dT_{n,\ell}+b_n\qquad(\ell\ge n).
\end{equation}
\end{lemma}
\begin{proof}
The gradients are bounded on the deterministic compact set $\chi$.
Their evaluations at $x^\nu$ are $\mathcal F_\nu$-measurable, so the
signed increments in~\eqref{eq:constraint-noise-tail} and
\eqref{eq:objective-noise-tail} are martingale differences.  If $G$ bounds
one of the relevant gradients, its increment has second moment at most
$G^2\sigma^2(\gamma^\nu)^2$.  Apply
Lemma~\ref{thm:l2-martingale} to these finitely many scalar series and
to each component of $\sum_\nu\gamma^\nu M_\nu$.

Writing $u^\nu:=x^{\nu+1}-x^\nu$, conditional unbiasedness gives
\[
\mathbb E[\|u^\nu\|^2\mid\mathcal F_\nu]
=(\gamma^\nu)^2\left(\|d(x^\nu)\|^2+
                 \mathbb E[\|M_\nu\|^2\mid\mathcal F_\nu]\right)
\le(B_d^2+\sigma^2)(\gamma^\nu)^2.
\]
Tonelli's theorem implies $\mathbb E\sum_\nu\|u^\nu\|^2<\infty$;
hence the nonnegative series is finite almost surely and its tails
$Q_n$ vanish.  Summing~\eqref{eq:stochiter} and using
\eqref{eq:direction-bound} proves~\eqref{eq:displacement-window}.
\end{proof}

\begin{lemma}[Convergence of values from interval descent]
\label{lem:interval-values}
Let $(z^\nu)$ lie in a compact set $K$ and satisfy
$z^{\nu+1}-z^\nu\to0$.  Let $V$ and $A\ge0$ be continuous on $K$.
Suppose that, for every $\ell>n$,
\begin{equation}\label{eq:abstract-interval}
V(z^\ell)-V(z^n)
\le-\sum_{\nu=n}^{\ell-1}\gamma^\nu A(z^\nu)
       +\eta_nT_{n,\ell}+e_n,
\end{equation}
where $\eta_n,e_n\ge0$ tend to zero, independently of the endpoint $\ell$,
and $\sum_\nu\gamma^\nu=\infty$.
If $V(\{z\in K:A(z)=0\})$ contains no interval of positive length, then
$V(z^\nu)$ converges.  Its limit belongs to
$V(\{z\in K:A(z)=0\})$.
\end{lemma}
\begin{proof}
If $A$ were bounded below by a positive constant on all of $K$, fixing a
sufficiently late $n$ in~\eqref{eq:abstract-interval} and letting
$\ell\to\infty$ would contradict the lower boundedness of $V$.
Thus the zero set of $A$ is nonempty.  Its image under $V$ is compact.

Suppose $\underline v:=\liminf_\nu V(z^\nu)$ is strictly smaller than
$\overline v:=\limsup_\nu V(z^\nu)$.  Choose
$v\in(\underline v,\overline v)$ outside that compact image.  There is
$h>0$ such that $[v-2h,v+2h]\subset(\underline v,\overline v)$ and this
closed interval is disjoint from the image.  Compactness gives an $a>0$
such that $A(z)\ge a$ whenever $z\in K$ and
$V(z)\in[v-2h,v+2h]$.

There are arbitrarily late upward crossings from at most $v-h$ to at least
$v+h$.  Let $\ell$ be the first index at or above $v+h$ in such a crossing,
and $n<\ell$ the last preceding index at or below $v-h$.
Uniform continuity of $V$ and $z^{\nu+1}-z^\nu\to0$ imply that
$V(z^\nu)\in[v-2h,v+2h]$ for all $n\le\nu<\ell$, once the crossing is
sufficiently late.  Therefore, for $\eta_n<a$,
\[
2h\le V(z^\ell)-V(z^n)
\le-(a-\eta_n)T_{n,\ell}+e_n\le e_n,
\]
contradicting $e_n\to0$.  Thus $V(z^\nu)$ converges.

If its limit were outside the compact image of the zero set, continuity
and compactness would give $A(z^\nu)\ge a>0$ eventually.  Fixing a late
$n$ in~\eqref{eq:abstract-interval} would again force
$V(z^\ell)\to-\infty$.  This proves the final assertion.
\end{proof}

\subsection{The ghost penalty over an interval}
For $\varepsilon>0$ define the population ghost penalty
\begin{equation}\label{eq:pendef}
W(x;\varepsilon):=F(x)+\frac{1}{\varepsilon}\phi(x).
\end{equation}
The parameter $\varepsilon$ is used only in the analysis and is never
computed by the algorithm.  A one-step argument cannot simply discard the
noise after taking the maximum: even if each $Y_i$ is centered,
$\mathbb E\max_iY_i$ need not be zero.  The next lemma retains the signed
noise in each row until it has been accumulated over an interval.

\begin{lemma}[Signed-noise interval estimates]\label{lem:ghost-interval}
Under the hypotheses of Lemma~\ref{lem:noise-tails}, fix a sample path on
the probability-one event in that lemma.  For every $\ell>n$,
\begin{equation}\label{eq:ghost-violation-interval}
\begin{aligned}
\phi(x^\ell)-\phi(x^n)
\le&-\sum_{\nu=n}^{\ell-1}\gamma^\nu\theta(x^\nu)
       +2r_nT_{n,\ell}+r_n+\frac{L_C}{2}Q_n.
\end{aligned}
\end{equation}
Choose $\varepsilon>0$ such that $\varepsilon^{-1}>\Lambda$, and let
\[
D(x):=c\|d(x)\|^2+
        (\varepsilon^{-1}-\Lambda)\theta(x).
\]
Then
\begin{equation}\label{eq:ghost-descent-interval}
\begin{aligned}
W(x^\ell;\varepsilon)-W(x^n;\varepsilon)
\le&-\sum_{\nu=n}^{\ell-1}\gamma^\nu D(x^\nu)
       +\frac{2r_n}{\varepsilon}T_{n,\ell}+e_n,\\
e_n:=&\ s_n+\frac{r_n}{\varepsilon}
            +\frac12\left(L_F+\frac{L_C}{\varepsilon}\right)Q_n
       \longrightarrow0.
\end{aligned}
\end{equation}
All error bounds are uniform in the endpoint $\ell$.
\end{lemma}
\begin{proof}
Write $C_0\equiv0$, $d^\nu=d(x^\nu)$, and
$u^\nu=x^{\nu+1}-x^\nu=\gamma^\nu(d^\nu+M_\nu)$.
For $i=1,\ldots,m$, Taylor's inequality, subproblem feasibility, and
$0<\gamma^\nu\le1$ give
\begin{equation}\label{eq:ghost-row-noise}
\begin{aligned}
C_i(x^{\nu+1})
&\le C_i(x^\nu)+\gamma^\nu\nabla C_i(x^\nu)^Td^\nu
          +\gamma^\nu\nabla C_i(x^\nu)^TM_\nu
          +\tfrac{L_C}{2}\|u^\nu\|^2\\
&\le(1-\gamma^\nu)C_i(x^\nu)+\gamma^\nu\kappa(x^\nu)
          +\gamma^\nu\nabla C_i(x^\nu)^TM_\nu
          +\tfrac{L_C}{2}\|u^\nu\|^2.
\end{aligned}
\end{equation}
The same inequality holds for the zero row, since $\kappa\ge0$.
Fix $n$, and for $\nu\ge n$ set
\[
S_{i,\nu}:=\sum_{j=n}^{\nu-1}\gamma^j\nabla C_i(x^j)^TM_j,
\qquad G_{i,\nu}:=C_i(x^\nu)-S_{i,\nu},
\qquad \Phi_\nu:=\max_{0\le i\le m}G_{i,\nu}.
\]
Here $S_{0,\nu}=G_{0,\nu}=0$.  Subtracting $S_{i,\nu+1}$
from~\eqref{eq:ghost-row-noise} cancels its current noise increment:
\[
G_{i,\nu+1}\le(1-\gamma^\nu)G_{i,\nu}
           +\gamma^\nu\kappa(x^\nu)-\gamma^\nu S_{i,\nu}
           +\tfrac{L_C}{2}\|u^\nu\|^2.
\]
Since $|S_{i,\nu}|\le r_n$ and the maximum of finitely many numbers
is 1-Lipschitz in their infinity norm,
\begin{equation}\label{eq:violation-comparison}
|\Phi_\nu-\phi(x^\nu)|\le r_n.
\end{equation}
Taking the maximum in the preceding row inequality therefore yields
\[
\begin{aligned}
\Phi_{\nu+1}-\Phi_\nu
&\le\gamma^\nu\bigl(\kappa(x^\nu)-\Phi_\nu\bigr)
          +\gamma^\nu r_n+\tfrac{L_C}{2}\|u^\nu\|^2\\
&\le-\gamma^\nu\theta(x^\nu)+2\gamma^\nu r_n
          +\tfrac{L_C}{2}\|u^\nu\|^2.
\end{aligned}
\]
Sum from $n$ to $\ell-1$, use $\Phi_n=\phi(x^n)$ and
$\phi(x^\ell)\le\Phi_\ell+r_n$, and bound the quadratic sum by $Q_n$.
This proves~\eqref{eq:ghost-violation-interval}.

For the smooth objective, Taylor's inequality and
\eqref{eq:convscfinter} give
\begin{equation}\label{eq:ghost-objective-interval}
\begin{aligned}
F(x^\ell)-F(x^n)
\le&\sum_{\nu=n}^{\ell-1}\gamma^\nu
       \bigl[-c\|d^\nu\|^2+\Lambda\theta(x^\nu)\bigr]\\
&+\sum_{\nu=n}^{\ell-1}\gamma^\nu\nabla F(x^\nu)^TM_\nu
       +\tfrac{L_F}{2}Q_n.
\end{aligned}
\end{equation}
Add $\varepsilon^{-1}$ times~\eqref{eq:ghost-violation-interval} and
use the uniform objective-noise bound $s_n$.
\end{proof}

\subsection{Almost-sure feasibility and stationarity}
The stronger conclusion below uses a condition on stationary objective
values.  It does not follow merely from continuity of the direction map.

\begin{assumption}[Weak Sard condition]\label{as:weak-sard}
With $\mathcal S$ the feasible KKT set of~\eqref{eq:prob}, the set
$F(\mathcal S\cap\chi)$ contains no interval of positive length.
\end{assumption}
This type of stationary-value condition is used in Lyapunov approaches to
stochastic approximation~\cite{benaim1999dynamics}.
For example, a finite or countable set of stationary objective values
satisfies this condition.  We state it explicitly rather than deducing it
from the smoothness assumptions, which alone do not provide the needed
conclusion.

\begin{theorem}[Almost-sure convergence to the KKT set]\label{th:conv}
Suppose Assumptions~\ref{as:lip}, \ref{as:dtilde}, and~\ref{as:bound} hold,
and let $(x^\nu)$ be generated by~\eqref{eq:stochiter} with
stepsizes~\eqref{eq:stepsizereq}.  Then, almost surely,
\begin{equation}\label{eq:feasibility-liminf}
\phi(x^\nu)\longrightarrow0,
\qquad
\liminf_{\nu\to\infty}\|d(x^\nu)\|=0.
\end{equation}
Every accumulation point is feasible, and at least one accumulation point
is a KKT point of~\eqref{eq:prob}.

If Assumption~\ref{as:weak-sard} also holds, then, almost surely,
\begin{equation}\label{eq:limsup-stationarity}
\limsup_{\nu\to\infty}\|d(x^\nu)\|=0,
\qquad
\operatorname{dist}(x^\nu,\mathcal S\cap\chi)\longrightarrow0.
\end{equation}
Moreover, $F(x^\nu)$ converges, and every accumulation point is a KKT
point with the same objective value.  These are convergence statements
for the stationarity residual and the distance to the KKT set; they do
not, without an additional isolation assumption, assert convergence of
the entire iterate sequence to one point.
\end{theorem}
\begin{proof}
Work on the probability-one event of Lemma~\ref{lem:noise-tails} and fix
$\varepsilon^{-1}>\Lambda$.  All arguments from this point are pathwise.

\paragraph{Feasibility.}
Apply Lemma~\ref{lem:interval-values} to $V=\phi$, $A=\theta$, and
\eqref{eq:ghost-violation-interval}.  Lemma~\ref{lem:slater} gives
$\{x\in\chi:\theta(x)=0\}=\{x\in\chi:\phi(x)=0\}$, whose image under
$\phi$ is $\{0\}$ if nonempty.  The interval inequality itself excludes
an empty zero set, as in the first paragraph of
Lemma~\ref{lem:interval-values}.  Hence $\phi(x^\nu)\to0$.
Continuity of the constraints makes every accumulation point feasible.

\paragraph{The liminf statement.}
Suppose $\|d(x^\nu)\|\ge\delta>0$ for all sufficiently large $\nu$.
Choose a fixed, sufficiently large $n$ so that
$2r_n/\varepsilon\le c\delta^2/2$.
Since $D(x)\ge c\|d(x)\|^2$,~\eqref{eq:ghost-descent-interval} implies
\[
W(x^\ell;\varepsilon)-W(x^n;\varepsilon)
\le-\tfrac12c\delta^2T_{n,\ell}+e_n\qquad(\ell>n).
\]
Letting $\ell\to\infty$ contradicts the lower boundedness of $W$ on
$\chi$.  Thus the liminf in~\eqref{eq:feasibility-liminf} is zero.
A subsequence with directions tending to zero has a convergent
subsequence of iterates by compactness.  Lemma~\ref{prop:dtheta} makes
its limit a KKT point.

\paragraph{Convergence of the penalty values under weak Sard.}
The function $D$ is continuous and nonnegative, and
Lemma~\ref{prop:dtheta} gives
\[
\{x\in\chi:D(x)=0\}=\mathcal S\cap\chi.
\]
On that set $\phi=0$, so $W(\mathcal S\cap\chi;\varepsilon)
=F(\mathcal S\cap\chi)$.  Assumption~\ref{as:weak-sard} and
Lemma~\ref{lem:interval-values}, applied to~\eqref{eq:ghost-descent-interval},
therefore imply convergence of $W(x^\nu;\varepsilon)$.
Since $\phi(x^\nu)\to0$, $F(x^\nu)$ converges to the same limit.

\paragraph{The limsup statement: excluding repeated excursions.}
The liminf result only guarantees arbitrarily small residuals along a
subsequence.  To obtain~\eqref{eq:limsup-stationarity}, we must also rule
out arbitrarily late returns to a residual bounded away from zero.
Suppose, to the contrary, that some $\delta>0$ and infinitely many starting
indices $n$ satisfy $\|d(x^n)\|\ge2\delta$.
Uniform continuity of $d$ on $\chi$ provides a radius $a>0$ such that
\[
\|x-y\|\le a,\quad x,y\in\chi
\quad\Longrightarrow\quad\|d(x)-d(y)\|\le\delta.
\]
Set $h:=a/(4B_d)>0$, and let $\ell(n)>n$ be the first index for which
$T_{n,\ell(n)}\ge h$.  This index exists because the stepsizes have
infinite sum.  Since $\gamma^\nu\to0$,
\[
h\le T_{n,\ell(n)}\le h+\sup_{\nu\ge n}\gamma^\nu\le2h
\]
for all sufficiently large $n$.  Using $b_n\to0$ and
\eqref{eq:displacement-window}, we obtain, simultaneously for
$n\le j\le\ell(n)$,
\[
\|x^j-x^n\|\le B_d T_{n,j}+b_n\le2B_dh+b_n<a.
\]
Consequently $\|d(x^j)\|\ge\delta$ throughout this interval.  For
sufficiently large selected $n$, take
$2r_n/\varepsilon\le c\delta^2/2$ and
$e_n\le c\delta^2h/4$.  The interval estimate now gives the fixed decrease
\[
\begin{aligned}
W(x^{\ell(n)};\varepsilon)-W(x^n;\varepsilon)
&\le-\tfrac12 c\delta^2T_{n,\ell(n)}+e_n\\
&\le-\tfrac14c\delta^2h<0.
\end{aligned}
\]
Both endpoints tend to infinity, so this contradicts convergence of the
penalty values.  Therefore $\limsup_\nu\|d(x^\nu)\|=0$.

Since the residual is nonnegative, the last statement is precisely
$d(x^\nu)\to0$.  Continuity and~\eqref{eq:zero-direction-kkt} then make
every accumulation point a KKT point.  If the distance in
\eqref{eq:limsup-stationarity} did not tend to zero, compactness would
provide a subsequence staying a positive distance from
$\mathcal S\cap\chi$ and converging to a point of that very set, a
contradiction.  The common objective value follows from convergence of
$F(x^\nu)$.
\end{proof}

\begin{remark}[What the interval estimate does and does not imply]
The factor $r_n$ tends to zero uniformly over the endpoint, but the product
$r_nT_{n,\ell}$ need not remain bounded as $\ell\to\infty$ for a fixed
$n$.  Therefore~\eqref{eq:ghost-descent-interval} does not establish
$\sum_\nu\gamma^\nu\|d(x^\nu)\|^2<\infty$ by summing to infinity.
The proof above uses eventual negativity for the liminf result, a
stationary-value crossing argument for convergence of $W$, and bounded
accumulated-stepsize intervals for the limsup result.  It requires neither
$\sum_\nu\gamma^\nu\|M_\nu\|<\infty$ nor vanishing conditional noise
variance.
\end{remark}

\begin{corollary}[Localization for almost-sure boundedness]
\label{cor:localized-convergence}
Consider the recursion~\eqref{eq:stochiter} with
stepsizes~\eqref{eq:stepsizereq}.  Suppose $F$ and $C_i$ have locally
Lipschitz continuous gradients on $\mathbb R^n$, $H$ is symmetric, continuous, and
positive definite at every point, and eMFCQ holds on $\mathbb R^n$.
Assume that the noise is integrable, conditionally centered,
$\mathbb E[M_\nu\mid\mathcal F_\nu]=0$, and has locally bounded
conditional second moments: for every integer $R\ge1$ there is a finite
constant $\sigma_R^2$ such that
\begin{equation}\label{eq:local-noise-bound}
\mathbb E[\|M_\nu\|^2\mid\mathcal F_\nu]\le\sigma_R^2
\quad\text{on }\{\|x^\nu\|\le R\},\qquad\text{for every }\nu.
\end{equation}
If $\sup_\nu\|x^\nu\|<\infty$ almost surely, then the feasibility,
liminf, and accumulation-point assertions in the first part of
Theorem~\ref{th:conv} remain valid.  If, in addition,
$F(\mathcal S\cap\overline B_R)$ contains no interval of positive length
for every integer $R\ge1$, where
$\overline B_R:=\{x:\|x\|\le R\}$, then almost surely
\[
\|d(x^\nu)\|\longrightarrow0,\qquad
\operatorname{dist}(x^\nu,\mathcal S)\longrightarrow0,
\]
$F(x^\nu)$ converges, and every accumulation point is a KKT point with
that limiting objective value.
\end{corollary}
\begin{proof}
For each integer $R\ge1$, define the stopping time
\[
\tau_R:=\inf\{\nu\ge0:\|x^\nu\|>R\},
\]
with $\inf\emptyset=\infty$.  The indicator
$\mathbf1_{\{\nu<\tau_R\}}$ is $\mathcal F_\nu$-measurable.
Consequently, the stopped raw increments
\[
\mathbf1_{\{\nu<\tau_R\}}\gamma^\nu M_\nu
\]
are vector martingale differences, with unconditional second moment at
most $\sigma_R^2(\gamma^\nu)^2$.  Multiplying additionally by
$\nabla F(x^\nu)^T$ or by $\nabla C_i(x^\nu)^T$ produces scalar
martingale differences with the same square-summability property, since
the gradients are bounded on $\overline B_R$.
Lemma~\ref{thm:l2-martingale} therefore gives convergence and uniform
Cauchy tails for every one of these stopped series.

The population direction is still bounded by $B_d$ everywhere because
its subproblem retains the box constraint.  Conditional centering and
\eqref{eq:local-noise-bound} give
\[
\mathbb E\sum_{\nu=0}^{\infty}
 \mathbf1_{\{\nu<\tau_R\}}\|x^{\nu+1}-x^\nu\|^2
\le(B_d^2+\sigma_R^2)\sum_{\nu=0}^{\infty}(\gamma^\nu)^2<\infty.
\]
Hence the stopped nonnegative series is finite almost surely as well.
Intersect these probability-one events over the finitely many series
and then over the countably many integers $R$.

On the resulting probability-one event, almost-sure boundedness supplies
an integer $R$ for which $\tau_R=\infty$.  Along that sample path every
stopped series just considered coincides with its original, unstopped
series.  Thus all the tail and displacement conclusions of
Lemma~\ref{lem:noise-tails} hold for this path.  Continuity and pointwise
positive definiteness of $H$ provide finite constants
$0<c_R\le\overline c_R$ on $\overline B_R$.  Local Lipschitz regularity
of the gradients provides finite Taylor constants there, and
Lemmas~\ref{lem:slater} and~\ref{prop:dtheta} provide the uniform Slater
margin, multiplier bound, and continuity of $d$ on this deterministic
ball.  All step segments lie in the ball because both endpoints do and
the ball is convex.

We can therefore apply the deterministic, pathwise interval argument in
Lemma~\ref{lem:ghost-interval} and the proof of
Theorem~\ref{th:conv}, with these constants and a fixed choice
$\varepsilon_R^{-1}>\Lambda_R$.  This gives the first set of conclusions,
and, under the stated weak Sard condition on the ball, the stronger
ones.  This reasoning applies the pathwise proof to the original
recursion on a no-exit event; it does not apply the convergence theorem
to an artificially frozen recursion after its exit time.
\end{proof}

\section{Unbiased estimation of the population direction}\label{s:estimator}\label{s:saa}

The convergence analysis requires a conditionally unbiased direction with a
uniformly bounded conditional second moment. An ordinary sample-average
subproblem does not, in general, provide such an estimator: taking the solution
of an optimization problem is a nonlinear operation. We use the antithetic
randomization construction of \cite{blanchet2019unbiased}. We first state the
precise rate needed for that construction and then prove a sufficient
condition. The rate is an additional regularity requirement; it does not follow
from strong convexity and a constraint qualification alone.

\subsection{Sample-average subproblems}

Fix $x\in\chi$. Write $A(x)=\nabla C(x)\in\mathbb R^{n\times m}$ and
$g(x)=\nabla F(x)$. One joint oracle call returns
\[
 Z_x=(g_x,b_x,\operatorname{vec} A_x,\operatorname{vech} Q_x),
 \qquad
 z_x:=\mathbb E Z_x
   =(g(x),C(x),\operatorname{vec} A(x),\operatorname{vech} H(x)).
\]
Here $g_x\in\mathbb R^n$, $b_x\in\mathbb R^m$, and
$A_x\in\mathbb R^{n\times m}$ are unbiased estimates of their population
counterparts. The symmetric curvature sample satisfies $Q_x\succeq cI$
almost surely, with the same $c>0$ for every $x$. The simple choice
$Q_x=H(x)$ is allowed. A raw Hessian of a nonconvex sample loss need not satisfy
this condition. If regularization or eigenvalue modification is used, $H(x)$
must be the expectation of the \emph{modified} curvature oracle.

At iteration $\nu$, conditional on $\mathcal F_\nu$, the oracle calls are
independent copies of $Z_{x^\nu}$. Components within a single call may be
dependent. All sample averages and optimization solution maps below are
assumed measurable. The sampling law and all bounds in this section are
understood conditionally in this sense when $x=x^\nu$.

For a coefficient tuple $z=(g,b,A,Q)$, set
\begin{align}
 q_z(d)&=g^Td+\tfrac12d^TQd,\label{eq:stoch_obj}\\
 a_z(d)&=b+A^Td,\label{eq:stoch_cons}\\
 v_z&=\max\{0,b_1,\ldots,b_m\},\notag\\
 r_z&=\min_{\|v\|_\infty\le\rho}
       \max\{0,b_1+A_1^Tv,\ldots,b_m+A_m^Tv\},\notag\\
 k_z&=(1-\lambda)v_z+\lambda r_z.\label{eq:kappa_coefficients}
\end{align}
The notation $A_i$ denotes the $i$th column of $A$.
Let $T(z)$ be the unique solution of
\begin{equation}\label{eq:subp1b}
 \min_{\|d\|_\infty\le\beta}q_z(d)
 \quad\text{subject to}\quad b+A^Td\le k_z e.
\end{equation}
The minimizer defining $r_z$ is feasible for~\eqref{eq:subp1b}, since
$0\le r_z\le v_z$, $r_z\le k_z$, and $\rho<\beta$. Thus the subproblem is
nonempty and compact, and $Q\succeq cI$ gives uniqueness. In particular,
\begin{equation}\label{eq:step_bound_saa}
 \|T(z)\|\le B:=\sqrt n\,\beta,
 \qquad T(z_x)=d(x).
\end{equation}

Given $M$ oracle calls, let $\bar Z_M=M^{-1}\sum_{j=1}^M Z_{x,j}$ and
$D_M(x)=T(\bar Z_M)$. In particular, the sampled relaxation is
\begin{equation}\label{eq:kappa_noisy_opt}
 \kappa_M(x)=(1-\lambda)\max_i\{(\bar b_{M,i})_+\}
 +\lambda\min_{\|v\|_\infty\le\rho}
       \max_i\{(\bar b_{M,i}+\bar A_{M,i}^Tv)_+\}.
\end{equation}
The sample average is taken \emph{before} the positive part, maximum, and
minimization. Averaging $(b_{x,j,i})_+$ instead would approximate
$\mathbb E[(b_{x,i})_+]$, which can differ from $(C_i(x))_+$, and would change
the population subproblem.

\subsection{Antithetic randomization}\label{ss:unbiased}

For an integer $\ell\ge0$, use $2^{\ell+1}$ joint samples and split them into
two halves of size $2^\ell$. Let $\bar Z^O_{2^\ell}$ and
$\bar Z^E_{2^\ell}$ be the half-sample averages. Their full-sample average is
exactly $(\bar Z^O_{2^\ell}+\bar Z^E_{2^\ell})/2$. Define
\begin{equation}\label{eq:antithetic_increment}
 \Delta_\ell(x)=T\!\left(\frac{\bar Z^O_{2^\ell}
                                  +\bar Z^E_{2^\ell}}2\right)
 -\frac{T(\bar Z^O_{2^\ell})+T(\bar Z^E_{2^\ell})}{2}.
\end{equation}
Each of the three solves recomputes its relaxation from its own coefficient
average. Consequently,
\begin{equation}\label{eq:telescoping_increment}
 \mathbb E\Delta_\ell(x)
 =\mathbb E D_{2^{\ell+1}}(x)-\mathbb E D_{2^\ell}(x).
\end{equation}

\begin{assumption}[Consistency and antithetic rate]\label{as:antithetic_rate}
For every $x\in\chi$, $\mathbb E\|D_M(x)-d(x)\|\to0$ as $M\to\infty$.
There are constants $K_\Delta<\infty$ and $s>1$, independent of $x$ and
$\ell$, such that
\begin{equation}\label{eq:antithetic_rate}
 \mathbb E\|\Delta_\ell(x)\|^2\le K_\Delta 2^{-s\ell}.
\end{equation}
\end{assumption}

Choose $a\in(1,s)$ and let
\[
 p_\ell=(1-2^{-a})2^{-a\ell},\qquad \ell=0,1,\ldots.
\]
Independently of the oracle samples, draw $L$ with
$\mathbb P(L=\ell)=p_\ell$. Using an independent single-sample baseline
$D_1^{(0)}(x)$, define
\begin{equation}\label{eq:estd}
 \widetilde d(x)=D_1^{(0)}(x)+\frac{\Delta_L(x)}{p_L}.
\end{equation}
The procedure uses one baseline solve and three coupled solves for the sampled
level. The level distribution is the same at every major iteration; there is
no deterministic increase in the batch size.

\begin{theorem}[Unbiasedness, second moment, and sampling cost]
\label{th:unbiased_estimator}
Under Assumption~\ref{as:antithetic_rate}, the estimator
\eqref{eq:estd} satisfies, for every $x\in\chi$,
\[
 \mathbb E\widetilde d(x)=d(x),\qquad
 \mathbb E\|\widetilde d(x)\|^2\le
 2B^2+\frac{2K_\Delta}
 {(1-2^{-a})(1-2^{-(s-a)})}=:K_{\rm est}.
\]
It uses a finite expected number of joint oracle calls, namely
\begin{equation}\label{eq:expected_oracle_cost}
 \mathbb E[1+2^{L+1}]
 =1+\frac{2(1-2^{-a})}{1-2^{1-a}}<\infty.
\end{equation}
At a random iterate, the same conclusions hold conditionally on
$\mathcal F_\nu$. Hence
$M_\nu=\widetilde d(x^\nu)-d(x^\nu)$ satisfies
Assumption~\ref{as:dtilde}, with a constant independent of $\nu$.
\end{theorem}

\begin{proof}
By Cauchy--Schwarz and~\eqref{eq:antithetic_rate},
\[
 \sum_{\ell=0}^{\infty}\mathbb E\|\Delta_\ell(x)\|
 \le\sqrt{K_\Delta}\sum_{\ell=0}^{\infty}2^{-s\ell/2}<\infty.
\]
Thus $\Delta_L/p_L$ is integrable, and summing its expectation is justified.
Equation~\eqref{eq:telescoping_increment} and consistency give
\begin{align*}
 \mathbb E\widetilde d(x)
 &=\mathbb E D_1(x)+\sum_{\ell=0}^{\infty}\mathbb E\Delta_\ell(x)\\
 &=\lim_{J\to\infty}\mathbb E D_{2^{J+1}}(x)=d(x).
\end{align*}
Moreover, by conditioning on $L$,
\[
 \mathbb E\left\|\frac{\Delta_L(x)}{p_L}\right\|^2
 =\sum_{\ell=0}^{\infty}
       \frac{\mathbb E\|\Delta_\ell(x)\|^2}{p_\ell}
 \le\frac{K_\Delta}
 {(1-2^{-a})(1-2^{-(s-a)})}.
\]
Combine this bound with $\|D_1^{(0)}\|\le B$ and
$\|u+v\|^2\le2\|u\|^2+2\|v\|^2$. Summing the geometric series
$1+2\sum_\ell2^\ell p_\ell$ proves~\eqref{eq:expected_oracle_cost}.
All bounds are uniform in $x$. Conditional on $\mathcal F_\nu$, $x^\nu$ is
fixed and the construction uses fresh randomness with the stated conditional
law, so the proof applies verbatim. Finally, conditional unbiasedness gives
\[
 \mathbb E[\|M_\nu\|^2\mid\mathcal F_\nu]
 =\mathbb E[\|\widetilde d(x^\nu)\|^2\mid\mathcal F_\nu]
       -\|d(x^\nu)\|^2\le K_{\rm est}.
\]
\end{proof}

\begin{remark}[Sampling effort, truncation, and numerical solves]
The finite-cost statement is about oracle calls and coefficient accumulation.
There are a fixed number of optimization solves per estimator, each with
dimensions $n$ and $m$ independent of $L$. A finite expected wall-clock cost
also requires a suitable bound on the cost of those solves; it is not implied
by~\eqref{eq:expected_oracle_cost} alone. The theorem uses exact subproblem
solutions and an \emph{uncapped} level. Although every realized level is finite
almost surely, the possible sample counts are unbounded.
For example, replacing the correction by
$\mathbf1_{\{L\le J\}}\Delta_L/p_L$ changes its expectation to
$\mathbb E D_{2^{J+1}}(x)-\mathbb E D_1(x)$ and generally leaves a nonzero
finite-level bias. Fixed solver tolerances and other clipping rules likewise
require a separate error analysis. A fixed base batch of size $b>1$ is allowed
by replacing $D_1$ by $D_b$ and $2^\ell$ by $b2^\ell$ throughout.
\end{remark}

\subsection{A sufficient condition for the antithetic rate}

The following elementary moment estimate replaces the need for a
large-deviation or central-limit argument.

\begin{lemma}[Fourth moment of a sample mean]\label{lem:fourth_sample_mean}
Let $Y_1,\ldots,Y_M\in\mathbb R^q$ be independent copies of a centered random
vector with $\mathbb E\|Y_1\|^4\le K_4$. Then
\[
 \mathbb E\left\|M^{-1}\sum_{j=1}^M Y_j\right\|^4
 \le\frac{3qK_4}{M^2}.
\]
\end{lemma}
\begin{proof}
For a scalar centered component $Y_{j,i}$, independence gives
\[
 \mathbb E\left(M^{-1}\sum_{j=1}^M Y_{j,i}\right)^4
 =\frac{M\mathbb E Y_{1,i}^4+
       3M(M-1)(\mathbb E Y_{1,i}^2)^2}{M^4}
 \le\frac{3\mathbb E Y_{1,i}^4}{M^2}.
\]
Use $(\sum_i u_i^2)^2\le q\sum_i u_i^4$ and
$\sum_i\mathbb E Y_{1,i}^4\le\mathbb E\|Y_1\|^4$.
\end{proof}

\begin{assumption}[Uniform local expansion of the complete step map]
\label{as:coefficient_expansion}
The coefficient oracle has a uniformly bounded centered fourth moment:
$\sup_{x\in\chi}\mathbb E\|Z_x-z_x\|^4\le K_4<\infty$.
For each $x$, let $\mathcal Z_x$ be a convex set containing $z_x$ and the
support of $Z_x$, on which the curvature is positive definite.
There are constants $\delta>0$, $K_T<\infty$, and linear maps $J_x$ with
$\|J_x\|\le K_T$, such that, whenever
$z_x+h\in\mathcal Z_x$ and $\|h\|\le\delta$,
\begin{equation}\label{eq:complete_step_expansion}
 T(z_x+h)=T(z_x)+J_xh+R_x(h),
 \qquad \|R_x(h)\|\le K_T\|h\|^2.
\end{equation}
All constants are independent of $x$.
\end{assumption}

The restriction to $\mathcal Z_x$ allows deterministic coefficients. For
example, with deterministic constraints only the objective and curvature
coordinates need be perturbed. Crucially, $T$ in
\eqref{eq:complete_step_expansion} includes the computation of $k_z$; regularity
of the final quadratic program with its right-hand side held fixed is not
enough.

\begin{theorem}[Quadratic antithetic rate]\label{th:quadratic_antithetic}
Assumption~\ref{as:coefficient_expansion} implies
Assumption~\ref{as:antithetic_rate} with $s=2$. Consequently, every choice
$1<a<2$ yields the estimator guarantees of
Theorem~\ref{th:unbiased_estimator}.
\end{theorem}

\begin{proof}
Write $K_{\rm av}=3qK_4$ for the constant in
Lemma~\ref{lem:fourth_sample_mean}. Sample-mean consistency and the local
expansion imply $D_M(x)\to d(x)$ in probability. Since both directions have
norm at most $B$, this also implies convergence in $L^1$. Alternatively, the
bound on $J_x$ and Markov's inequality give this convergence uniformly in $x$.

Fix $M=2^\ell$ and set
$h_O=\bar Z^O_M-z_x$, $h_E=\bar Z^E_M-z_x$,
$h_F=(h_O+h_E)/2$. On the event
$G=\{\|h_O\|\le\delta,\ \|h_E\|\le\delta\}$, all three expansions
\eqref{eq:complete_step_expansion} apply. The constant and linear terms cancel
\emph{exactly}, so
\[
 \Delta_\ell=R_x(h_F)-\tfrac12R_x(h_O)-\tfrac12R_x(h_E),
 \qquad
 \|\Delta_\ell\|\le K_T(\|h_O\|^2+\|h_E\|^2).
\]
It follows that
\[
 \mathbb E[\|\Delta_\ell\|^2\mathbf1_G]
 \le\frac{4K_T^2K_{\rm av}}{M^2}.
\]
On the complementary event, use the global bound
$\|\Delta_\ell\|\le2B$ and the fourth-moment estimate:
\[
 \mathbb P(G^c)\le\frac{2K_{\rm av}}{\delta^4M^2},\qquad
 \mathbb E[\|\Delta_\ell\|^2\mathbf1_{G^c}]
 \le\frac{8B^2K_{\rm av}}{\delta^4M^2}.
\]
Thus~\eqref{eq:antithetic_rate} holds with $s=2$ and
$K_\Delta=K_{\rm av}(4K_T^2+8B^2\delta^{-4})$. This establishes the required mean-square rate directly from fourth moments.
\end{proof}

\begin{remark}[How the expansion can be verified]
One sufficient situation is that, on a neighborhood of the compact set of
population coefficients (restricted to their stochastic coordinates), $k_z$
is twice continuously differentiable and the quadratic-program active set is
locally stable. Uniformly positive active multipliers, uniformly positive
inactive slacks, and uniformly linearly independent active constraint gradients
give a concrete version of the latter condition. Indeed, for the active
constraint matrix $B_I(z)$, including any active box faces, the primal-dual
solution solves
\[
 \begin{pmatrix}Q&B_I(z)^T\\B_I(z)&0\end{pmatrix}
 \begin{pmatrix}d\\\mu_I\end{pmatrix}
 =\begin{pmatrix}-g\\r_I(z)\end{pmatrix}.
\]
Positive definiteness of $Q$ and full row rank of $B_I(z)$ make this matrix
nonsingular. Smooth inversion, preserved slack and multiplier signs, and
bounded derivatives on a neighborhood of the compact population-coefficient
set give~\eqref{eq:complete_step_expansion}. The constants must be uniform up
to accumulation points. Pointwise strict complementarity only at the sampled
iterates does not supply such constants.
\end{remark}

\subsection{Why constraint regularity alone does not suffice}

The next example shows why the additional antithetic-rate condition cannot be
removed from the general stochastic-constraint result. It concerns the
particular estimator~\eqref{eq:estd}; it is not an impossibility result for every
possible unbiased estimation method.

\begin{example}[A kink in the sampled relaxation]\label{ex:kappa_kink}
At $x=0$, take $F(x)=\tfrac12(x-1)^2$, $C(x)=x$, and the unbiased constraint
oracle $c(x,\zeta)=x+\zeta$, where $\zeta\sim\mathcal N(0,1)$.
The objective gradient, constraint gradient, and curvature are deterministic.
Choose $\beta=2$, $\rho=1/2$, and any $\lambda\in(0,1)$.
The coefficient $b$ is the sampled constraint mean. Then
\[
 k(b)=(1-\lambda)b_++\lambda(b-\rho)_+,
 \qquad
 T(b)=\min\{1,k(b)-b\}.
\]
The lower box face cannot be active because $k(b)-b\ge-\lambda\rho>-\beta$.
For $|b|<1/2$,
\begin{equation}\label{eq:kink_step}
 T(b)=-b+(1-\lambda)b_+.
\end{equation}
At the population coefficient $b=0$, the subproblem is
$\min\{\tfrac12d^2-d:d\le0,\ |d|\le2\}$.
Its solution is $d=0$, its active gradient is $1$, and its multiplier is $1$.
Thus LICQ and strict complementarity hold. They also hold for almost every
sampled coefficient, and all coefficient moments exist.

Nevertheless, with $M=2^\ell$,
\begin{equation}\label{eq:kink_rate}
 \mathbb E|\Delta_\ell|^2=\Theta(M^{-1})=\Theta(2^{-\ell}).
\end{equation}
To prove the lower bound, write the two half-sample means as
$b_O=U/\sqrt M$ and $b_E=V/\sqrt M$, where $U,V$ are independent standard
normal variables. For all sufficiently large $M$, the event
$\{1\le U\le2,\ -2\le V\le-1\}$ lies inside the neighborhood where
\eqref{eq:kink_step} applies. Its probability is a fixed positive number, and
on this event
\[
 |\Delta_\ell|
 =\frac{1-\lambda}{2\sqrt M}\min\{U,-V\}
 \ge\frac{1-\lambda}{2\sqrt M}.
\]
For the upper bound, $T$ is globally Lipschitz with constant $1$, whence
$|\Delta_\ell|\le|b_O-b_E|/2$ and
$\mathbb E|\Delta_\ell|^2\le1/(2M)$.

No positive level probability sequence $(p_\ell)$ can provide both finite
second moment and finite expected sampling cost for this single-term
construction. Indeed, those properties would require both
$\sum_\ell\mathbb E|\Delta_\ell|^2/p_\ell<\infty$ and
$\sum_\ell p_\ell2^\ell<\infty$. Cauchy--Schwarz would then imply
\[
 \sum_\ell\sqrt{\mathbb E|\Delta_\ell|^2\,2^\ell}<\infty,
\]
contradicting~\eqref{eq:kink_rate}. The bounded baseline in
\eqref{eq:estd} cannot remove an infinite second moment of its correction.
Thus even light-tailed constraint noise and a regular final subproblem do not
establish the finite-variance, finite-expected-cost claim at a feasibility
boundary. The guarantees of this section apply when
Assumption~\ref{as:antithetic_rate}, or a sufficient condition such as
Assumption~\ref{as:coefficient_expansion}, is satisfied.
\end{example}

\section{Numerical results}\label{s:numerics}
We report computer-vision, inverse-problem, and offline score-auditing
experiments. The figures report finite-run empirical results.

\paragraph{Relation to the convergence theory.}
The convergence result concerns the population problem with inequality
constraints and a conditionally unbiased direction estimator satisfying the
stated moment and regularity assumptions. The experiments also include practical
variants whose assumptions are not established by the numerical results.
In particular, the CIFAR configurations below use $\rho=\beta$, whereas the
theory requires $0<\rho<\beta$; the score-auditing examples use explicit
equalities and a single-batch stochastic direction. These examples illustrate
numerical behavior and do not constitute applications of the population
convergence theorem. A fixed cap on the multilevel randomization, when used,
likewise changes the estimator and does not inherit exact unbiasedness.
The upper bounds in the examples are put in the form $C_i(x)\le0$ by subtracting their right-hand sides.
The reported residuals distinguish progress in objective value, feasibility,
and stationarity; none of these alone certifies the other two.

\subsection{CIFAR--10 classification}
We evaluated our stochastic approximation algorithm on the \texttt{CIFAR-10} dataset~\cite{krizhevsky2009learning}. The model is a neural network with one sigmoid hidden layer and a linear output layer.
\[
\sigma(t)=\frac{1}{1+e^{-t}}.
\]
To keep the computational burden moderate, each image was first converted to grayscale, resized to $16\times16$ resolution and then flattened. The hidden layer contained $32$ neurons, yielding a parameter vector of dimension $
256 \times 32 + 32 + 32 + 1 = 8257.
$

The dataset consists of $60{,}000$ images across $10$ classes. In our experiments, we used a subset of $2000$ samples and solved a binary classification problem for three different cases described below.

For a parameter vector $x\in\mathbb R^{8257}$, let
$\ell(x;a,y)$ be the binary cross-entropy loss with logits for an image
$a$ and its binary label $y$. Objectives and constraints below are empirical
averages of this loss over the indicated sample sets, with the specified
quadratic regularization where applicable. The objective and constraint
sample sets need not coincide.

For the optimization algorithm, the curvature model was set to the identity, and the quadratic subproblems were solved using OSQP~\cite{osqp}, which is well suited for the moderately large parameter dimension arising in this setting. For these CIFAR experiments, the reported settings are
\[
\beta = 0.5, \qquad \rho = 0.5, \qquad \lambda = 0.5, \qquad \tau = 1.
\]
The parameter $\beta$ controls the box constraint in the quadratic subproblem, $\rho$ controls the box constraint in the feasibility-relaxation linear program for computing $\kappa$, $\lambda$ is the averaging weight in the definition of the relaxation parameter $\kappa$, and $\tau$ is the curvature-model scaling parameter.

Across all parameter configurations, we performed $21$ independent trials and reported the median together with the interquartile range, namely the $25\%$ and $75\%$ quantiles, for both the objective values and the constraint violations over the iteration horizon.

\paragraph{Deterministic convex constraints.}

In this experiment, the objective is given by the binary cross-entropy loss with logits. For a sample $\xi=(a,y)$, where $a$ denotes the input and $y\in\{0,1\}$ the associated binary label, we define
\[
f(x,\xi)=\ell(x;\xi),
\]
where $\ell(x;\xi)$ is the binary cross-entropy loss between the network prediction and the label $y$.

The constraint is deterministic and is imposed directly on the parameter vector. Writing the parameter vector as $x=(W,b)$, where $W$ collects the weight matrices and $b$ collects the bias vectors, we consider the ellipsoidal constraint
\[
\frac{\|W\|_2^2}{a_w^2}+\frac{\|b\|_2^2}{a_b^2}\le c.
\]
In the experiments, we use
\[
a_w=2,\qquad a_b=1,\qquad c=5.
\]
Equivalently, the constraint function may be written as
\[
C(x)=\frac{\|W\|_2^2}{a_w^2}+\frac{\|b\|_2^2}{a_b^2},
\qquad \text{with} \qquad C(x)\le 5.
\]

Thus, the optimization minimizes the classification loss while constraining the parameters to remain inside a fixed ellipsoidal region, allowing different scaling for weights and biases. Since this constraint is deterministic, it does not depend on the sampled data $\xi$ and therefore serves as a direct geometric restriction on the feasible set.

\begin{figure}
    \centering
    \includegraphics[width=0.32\textwidth]{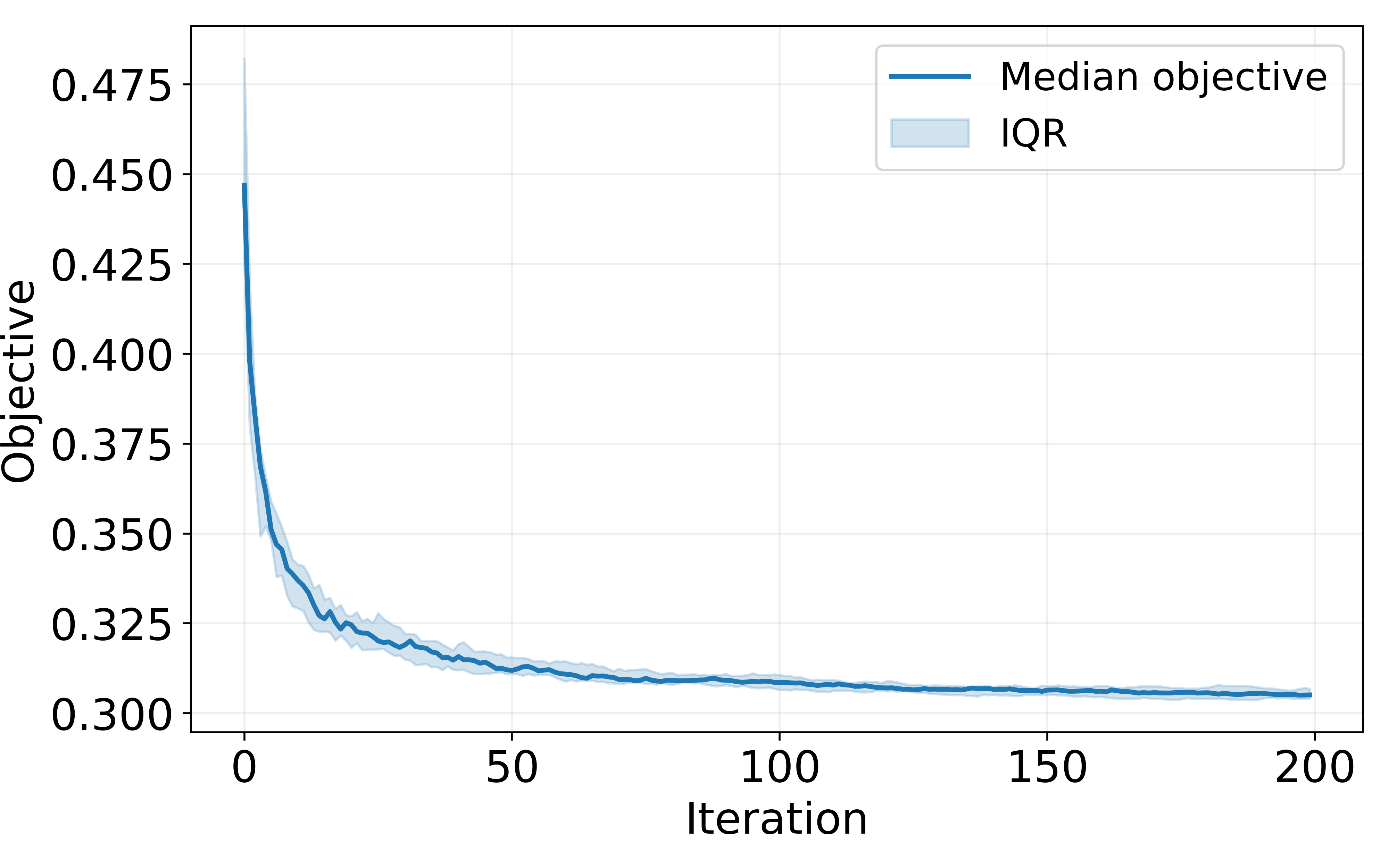}%
    \hfill%
    \includegraphics[width=0.32\textwidth]{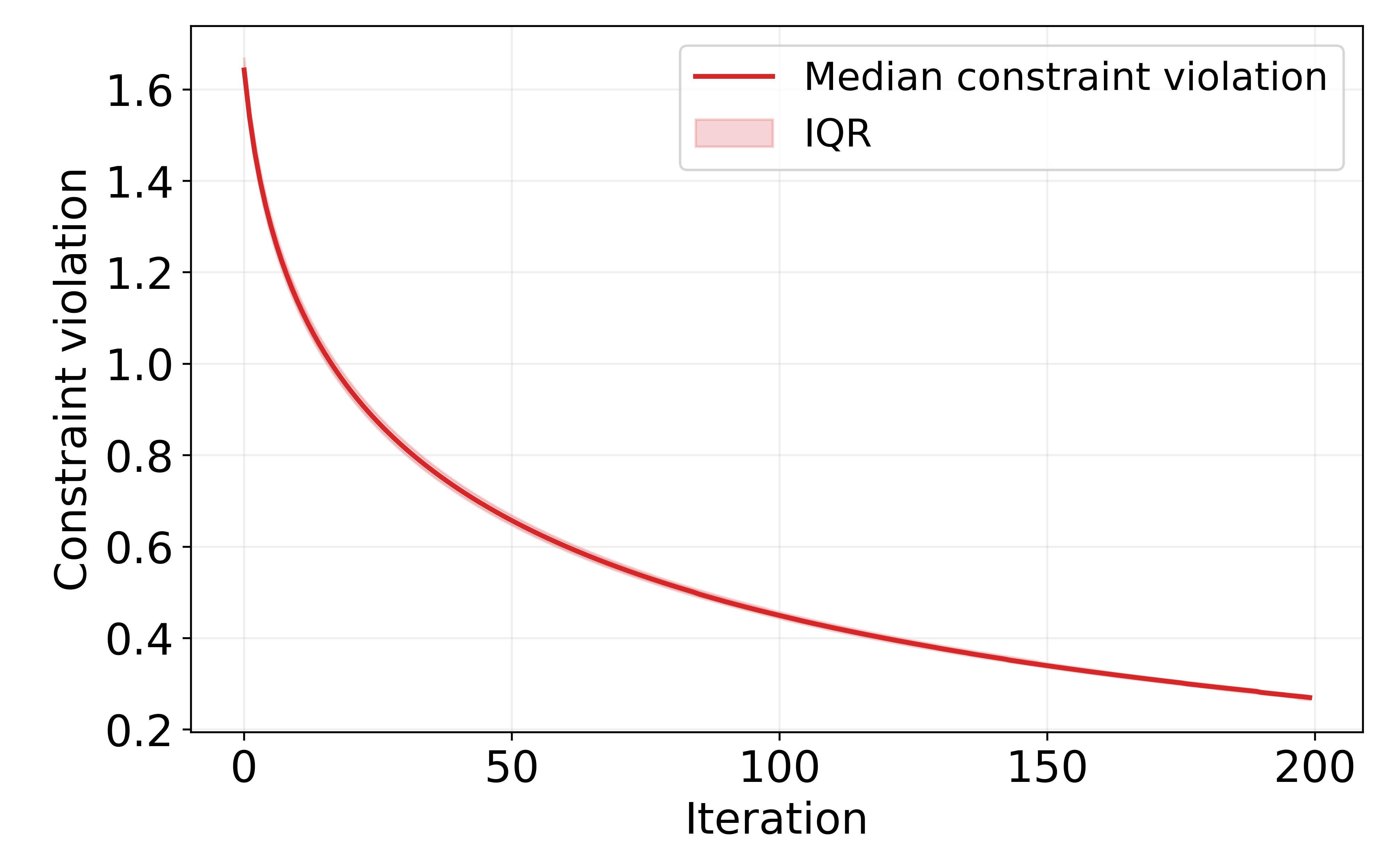}%
    \hfill%
    \includegraphics[width=0.32\textwidth]{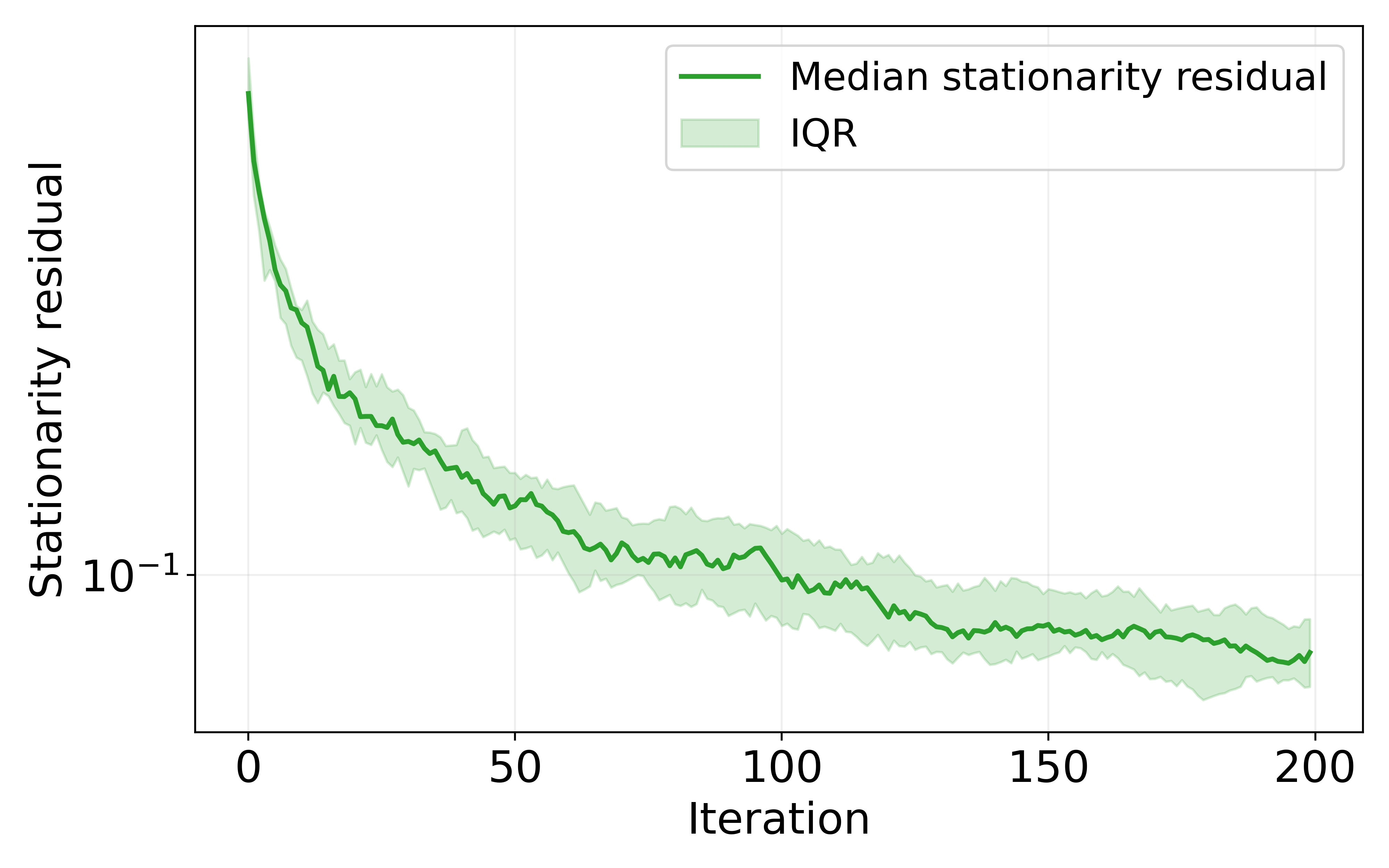}
    \caption{The plots show representative results for CIFAR deterministic convex constraints problem \textbf{Left}: The objective trajectory. \textbf{Middle} The constraint violation trajectory. \textbf{Right}: The fitted-multiplier stationarity residual. The shaded regions in the objective and constraint plots represent the interquartile range, and the solid line represents the median trajectory.}
    \label{fig:ellconfig}
\end{figure}

\paragraph{Stochastic constraints: training and validation.}

In this experiment, the objective and the constraint are given by the same binary cross-entropy loss with logits, but they are evaluated on two different sample sets corresponding to a training set and a validation set, respectively. For a training sample $\xi_{\mathrm{tr}}=(a,y)$ and a validation sample $\xi_{\mathrm{val}}=(a',y')$, we define
\[
f(x,\xi_{\mathrm{tr}})=\ell(x;\xi_{\mathrm{tr}}),
\qquad
c(x,\xi_{\mathrm{val}})=\ell(x;\xi_{\mathrm{val}}),
\]

where $\ell(x;\cdot)$ denotes the binary cross-entropy loss between the network prediction and the associated binary label.

The resulting stochastic optimization problem therefore takes the form
\[
\min_x \ \mathbb{E}\big[f(x,\xi_{\mathrm{tr}})\big]
\qquad
\text{subject to}
\qquad
\mathbb{E}\big[c(x,\xi_{\mathrm{val}})\big]\leq 0.20.
\]

Accordingly, the method aims to minimize the expected training loss while enforcing that the expected validation loss remains below a prescribed threshold. In this way, the constraint controls loss on the fixed validation set. Because that set is used during optimization, this is not an independent test of population generalization.

\begin{figure}
    \centering
    \includegraphics[width=0.32\textwidth]{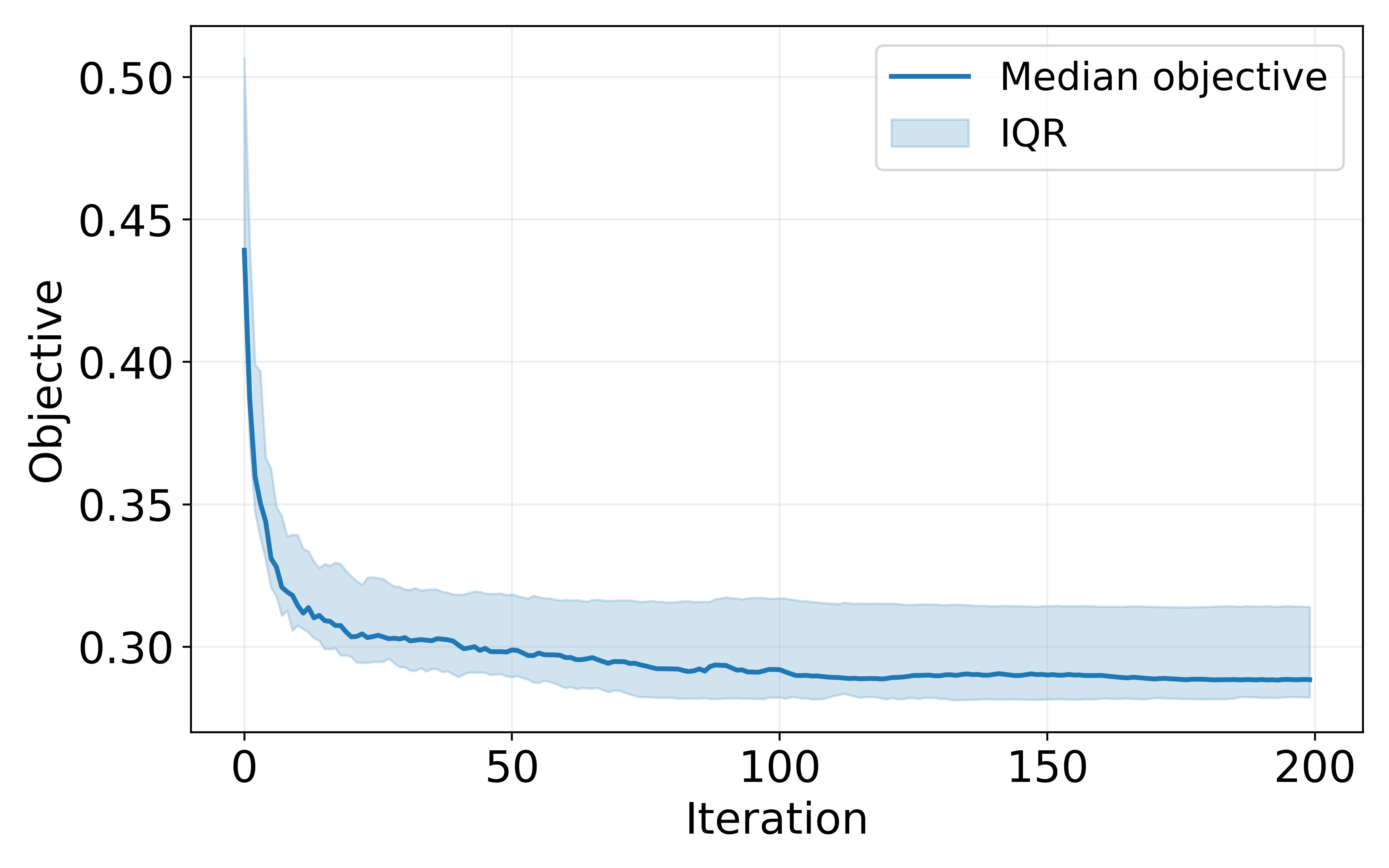}%
    \hfill%
    \includegraphics[width=0.32\textwidth]{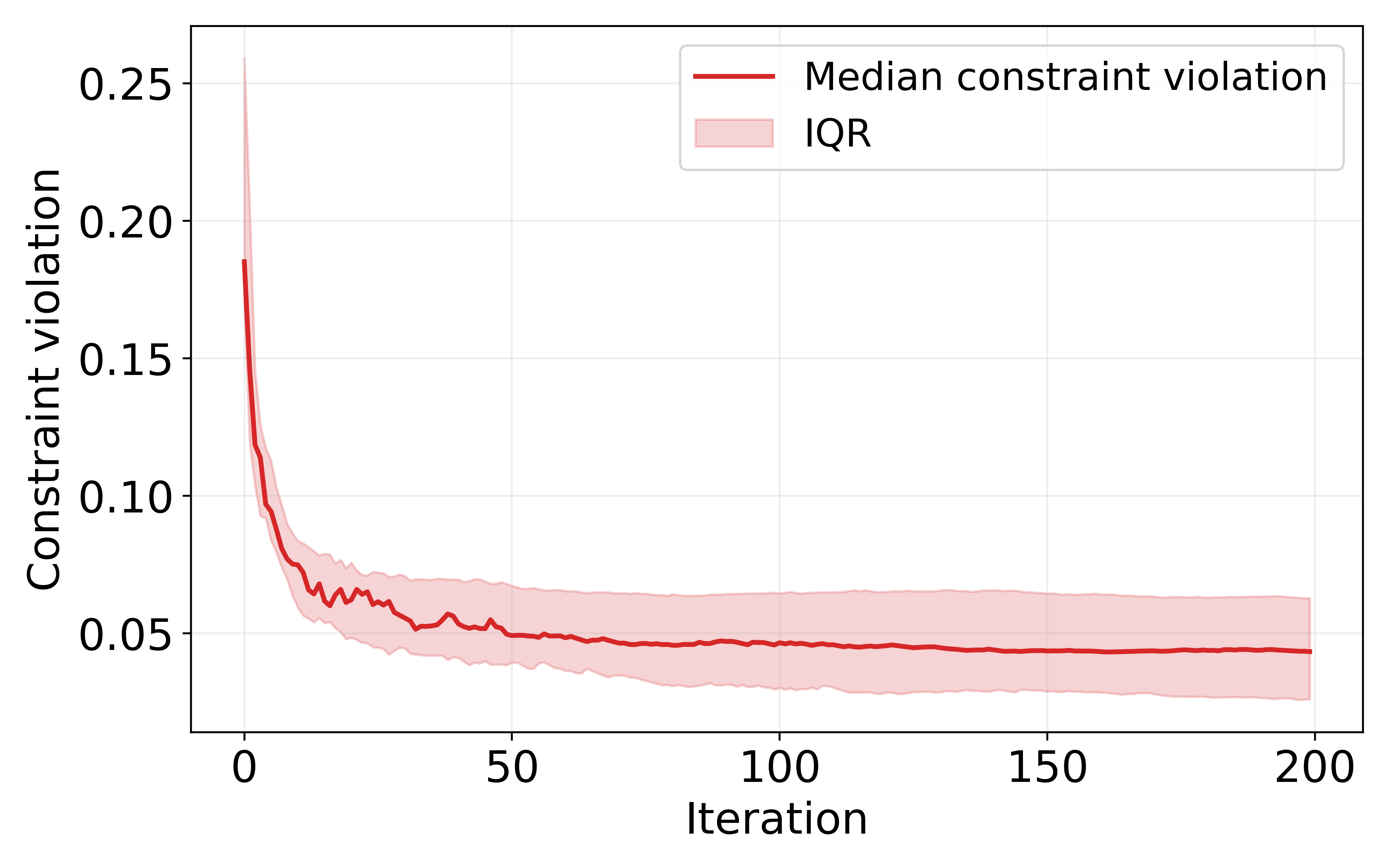}%
    \hfill%
    \includegraphics[width=0.32\textwidth]{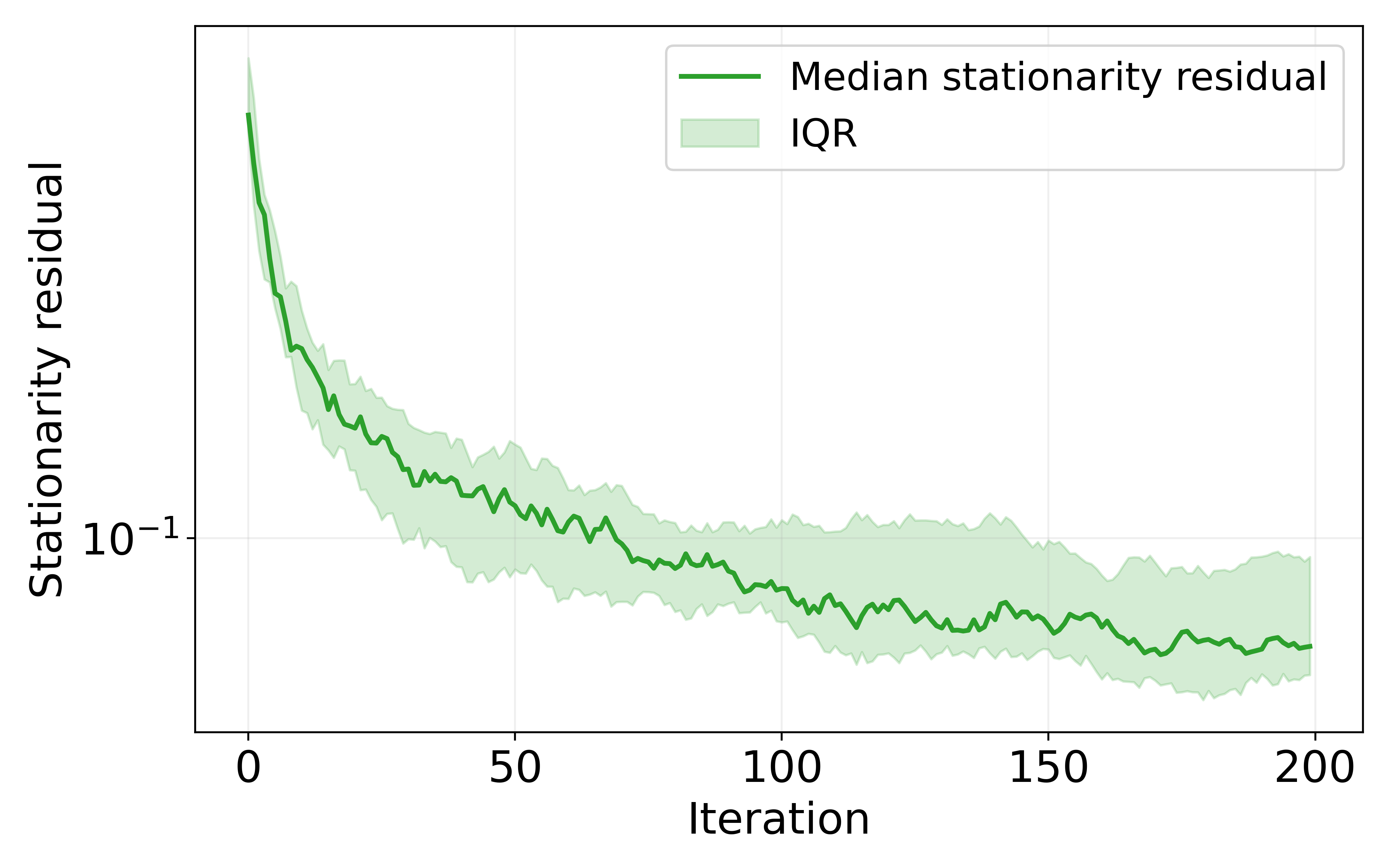}
    \caption{The plots show representative results for CIFAR training and validation problem \textbf{left}: The objective trajectory. \textbf{Middle} The constraint violation trajectory. \textbf{Right}: The fitted-multiplier stationarity residual. The shaded regions in the objective and constraint plots represent the interquartile range, and the solid line represents the median trajectory.}
    \label{fig:trainvalid}
\end{figure}

\paragraph{Stochastic constraints: a second training objective.}

In this experiment, the objective and the constraint are both defined through binary cross-entropy losses with logits, but they correspond to the training losses over two different target classes in the training set.
\[
f(x,\xi)=\ell_{\mathrm{obj}}(x;\xi)+0.01\|x\|_2^2,
\]
where $\ell_{\mathrm{obj}}(x;\xi)$ is the binary cross-entropy loss associated with the primary target class "dog".

The constraint is defined analogously, but with respect to a secondary target class "automobile":
\[
c(x,\xi)=\ell_{\mathrm{con}}(x;\xi)+0.01\|x\|_2^2,
\]

where $\ell_{\mathrm{con}}(x;\xi)$ is the corresponding binary cross-entropy loss for the constrained class. Feasibility is enforced through
\[
\mathbb{E}[c(x,\xi)]\leq 0.20.
\]

Hence, the optimization problem reduces the regularized classification loss for one target class while simultaneously controlling the regularized classification loss for another target class. Since both terms share the same parameter vector, the method is required to balance performance on the primary objective against the constrained behavior associated with the secondary class.

\begin{figure}
    \centering
    \includegraphics[width=0.32\textwidth]{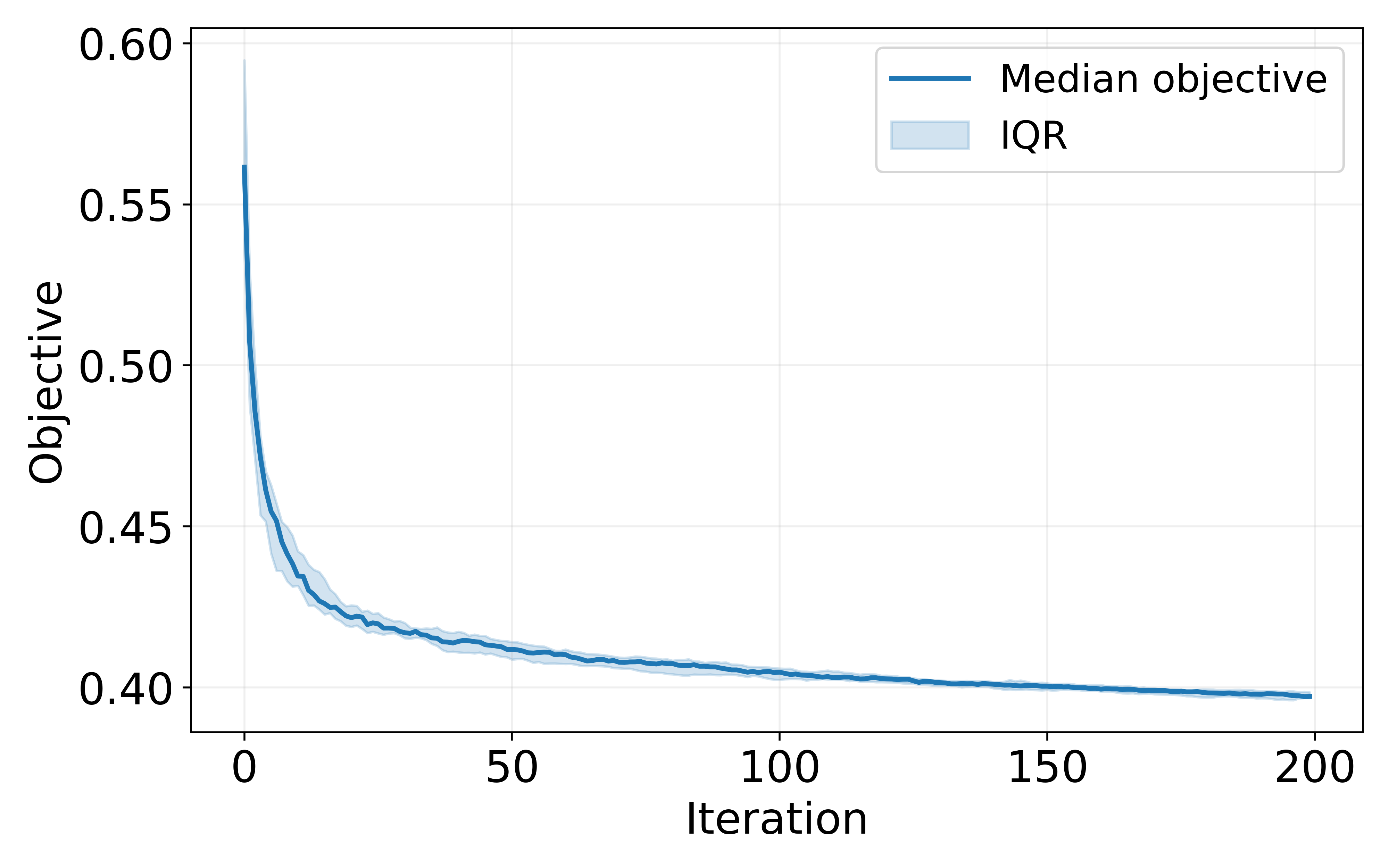}%
    \hfill%
    \includegraphics[width=0.32\textwidth]{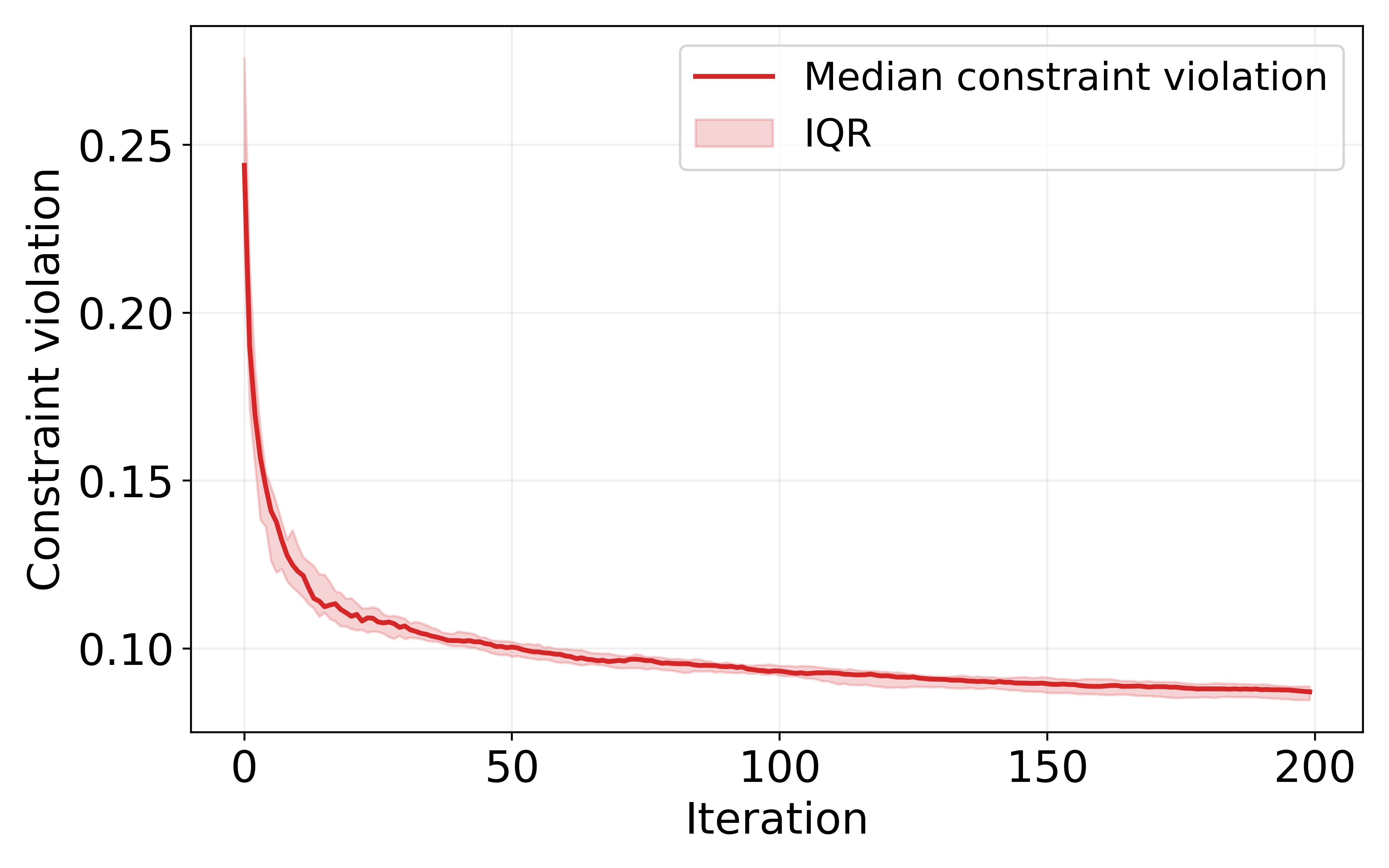}%
    \hfill%
    \includegraphics[width=0.32\textwidth]{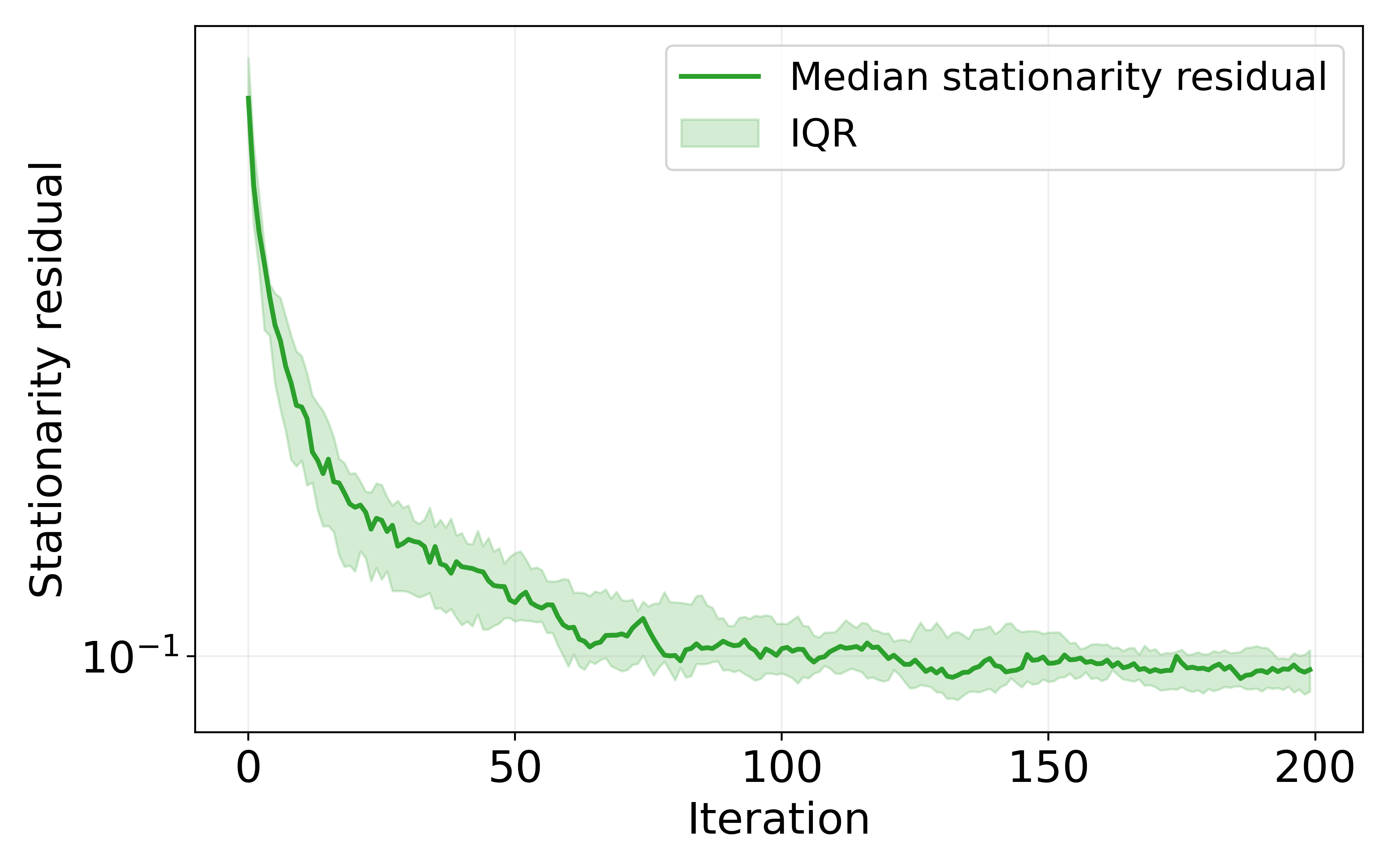}
    \caption{Representative results for CIFAR training constraint problem. \textbf{left}: The objective trajectory. \textbf{Middle} The constraint violation trajectory. \textbf{Right}: The fitted-multiplier stationarity residual. The shaded regions in the objective and constraint plots represent the interquartile range, and the solid line represents the median trajectory.}
    \label{fig:2obj}
\end{figure}

\paragraph{Analysis of the results.}
For each experiment, the three plots show:
\begin{enumerate}
    \item \textbf{Objective:}
    The median empirical objective across seeds, with shading between the
    $25$th and $75$th percentiles.

    \item \textbf{Constraint violation:}
    The median positive constraint violation,
    \[
        r_{\mathrm{prim}}(x)=\max\{0,C(x)\},
    \]
    with the corresponding interquartile range. Here, $C(x)\leq 0$ denotes
    the constraint after subtracting its prescribed bound.

    \item \textbf{Stationarity diagnostic:}
    The median stationarity residual,
    \[
        r_{\mathrm{stat}}(x)
        =\left\|\nabla F_N(x)+\widehat{\lambda}(x)\nabla C_N(x)\right\|_2,
    \]
     with its
    interquartile range. Here $F_N$ and $C_N$ are the empirical objective and the empirical
    constraint after subtracting its bound. The nonnegative multiplier is fitted
    at each evaluated iterate according to
    \[
        \widehat{\lambda}(x)
        \in \operatorname*{arg\,min}_{\lambda\geq 0}
        \left\{
            \left\|\nabla F_N(x)+\lambda\nabla C_N(x)\right\|_2^2
            +\left|\lambda C_N(x)\right|^2
        \right\}.
    \]
The complementarity term enters the multiplier fit but is not plotted as a
    separate residual. Hence the stationarity panel is not a complete KKT certificate.
\end{enumerate}
Across the three CIFAR--10 experiments, Ghost-SA exhibits a rapid initial
decrease in the median objective, followed by slower improvement over the
remaining iterations. Across 21 seeds and 200 iterations, the median objective
decreases from approximately $0.562$ to $0.397$ in the two-objective setting,
from $0.439$ to $0.288$ in the train--validation setting, and from $0.447$ to
$0.305$ in the elliptic setting. The median constraint violation also decreases,
from approximately $0.244$ to $0.087$, $0.185$ to $0.043$, and $1.643$ to
$0.270$, respectively. The elliptic experiment shows a particularly steady
reduction in violation with narrow interquartile bands, whereas the
train--validation experiment retains greater variability across seeds.

The stationarity diagnostics provide additional evidence of progress toward
the first-order necessary conditions. The median stationarity residual
decreases from approximately $0.677$ to $0.096$ in the two-objective experiment,
from $0.644$ to $0.062$ in the train--validation experiment, and from $0.667$ to
$0.074$ in the elliptic experiment. Nevertheless, both stationarity and primal
feasibility residuals remain nonzero at termination. These results therefore
demonstrate improvement toward KKT satisfaction within the iteration budget,
rather than convergence to a feasible KKT point. All reported diagnostics use
the fixed empirical dataset of 2000 images; the shaded bands describe variability
across seeds.

All three plots use iteration number on the horizontal axis.

\subsection{Catastrophic forgetting}

Catastrophic forgetting is a well-known problem in machine learning:
when a neural network is trained sequentially on different tasks,
learning the new task can reduce its performance on previously learned
tasks. We study this effect using two synthetic binary classification
problems, first learning from a circular data distribution $D_1$
and then a rectangular data distribution $D_2$.

Each task contains two equally represented classes: an inner class
with label $t=1$ and an outer class with label $t=0$.
For the circular distribution $D_1$, the class-dependent radii follow
\[
R_1\mid t=1
\sim\operatorname{clip}\!\left(\mathcal{N}(0.75,0.25^2),\,0.5,\,1.5\right),
\qquad
R_1\mid t=0
\sim\operatorname{clip}\!\left(\mathcal{N}(1.25,0.25^2),\,0.5,\,1.5\right).
\]
Thus, both Gaussian distributions have standard deviation $0.25$,
with means $0.75$ and $1.25$ for the inner and outer classes,
respectively. Sampled radii below $0.5$ are set to $0.5$, and those
above $1.5$ are set to $1.5$. We independently sample
$\phi\sim\mathcal{U}(0,2\pi)$ and define
\[
u_1=(R_1\cos\phi,\;R_1\sin\phi).
\]
These samples form two overlapping circular bands.

For the rectangular distribution $D_2$, the normalized radii follow
\[
R_2\mid t=1\sim\mathcal{U}(0.5,1.2),
\qquad
R_2\mid t=0\sim\mathcal{U}(0.8,1.5).
\]
We independently sample a point $v$ uniformly along the perimeter
of a rectangle with half-width $1.4$ and half-height $0.7$, and define
\[
u_2=R_2v.
\]

This produces two overlapping rectangular bands, as shown in
Figure~\ref{fig:input_dist}. Labels are determined by the class
from which each point is sampled.
The overlap between the class distributions makes classification
uncertain.

The nominal decision boundaries are
\[
x^2+y^2=1
\qquad\text{and}\qquad
\max\left\{\frac{|x|}{1.4},\frac{|y|}{0.7}\right\}=1.
\]
For the rectangle, the two classes have equal density throughout
the overlapping radial interval $[0.8,1.2]$, so the nominal
boundary is not a unique optimal decision boundary.
Each task contains $5{,}000$ samples, divided into
$3{,}000$ training, $1{,}000$ validation, and $1{,}000$ test samples,
with balanced classes in each split.

\begin{figure}
    \centering
    \includegraphics[width=0.8\linewidth]{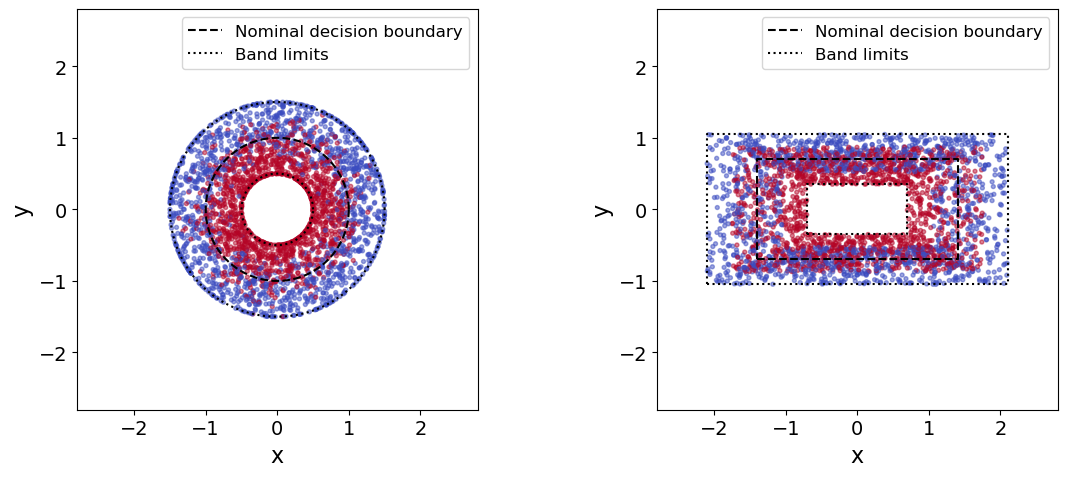}
    \caption{\textbf{Left:} The circular data distribution ($D_1$),
    whose inner and outer classes have clipped Gaussian radial
    distributions centered at $0.75$ and $1.25$, respectively.
    \textbf{Right:} The rectangular data distribution ($D_2$),
    whose classes have uniform radial distributions over
    $[0.5,1.2]$ and $[0.8,1.5]$.
    The overlapping bands create uncertainty between the classes.
    Dashed curves indicate the nominal decision boundaries.}
    \label{fig:input_dist}
\end{figure}

We use a neural network with two inputs, two hidden layers of
32 neurons each, and a single sigmoid output:
\[
p_\theta(u)
=
\sigma\!\left(
W_3\tanh\!\left(W_2\tanh(W_1u+b_1)+b_2\right)+b_3
\right).
\]
For a dataset $D$, its binary cross-entropy loss is
\[
L_D(\theta)
=
-\frac{1}{|D|}
\sum_{(u,t)\in D}
\left[
t\log p_\theta(u)
+(1-t)\log\!\left(1-p_\theta(u)\right)
\right].
\]
We start by learning the circle task using the Adam
optimizer, obtaining a trained model $M_{\theta_1}$ while monitoring
its circle validation loss and accuracy. We then initialize from
$\theta_1$ and train on the rectangle task. During this
second stage, we monitor validation loss and accuracy on both
datasets. As shown in Figure~\ref{fig:adam_ghost_validation} (top), learning the
rectangle reduces performance on the previously learned circle:
circle validation accuracy falls from $84.4\%$ to $71.1\%$, while
rectangle validation accuracy reaches $72.7\%$.

We then repeat rectangle training using Ghost, starting from the
same model $M_{\theta_1}$ and imposing an explicit constraint on
circle validation loss:
\[
\min_\theta L_{D_{2,\mathrm{train}}}(\theta)
\quad\text{subject to}\quad
L_{D_{1,\mathrm{val}}}(\theta)
\leq
\alpha L_{D_{1,\mathrm{val}}}(\theta_1),
\qquad \alpha>0.
\]
This constraint limits the deterioration permitted on the first task
while the model learns the second task.

For the run with $\alpha=1.2$, the circle validation loss settles
near $0.423$, below the bound of $0.424$, as shown in
Figure~\ref{fig:adam_ghost_validation} (bottom). Ghost retains a circle validation
accuracy of $80.1\%$, compared with Adam's $71.1\%$ after task switching.
However, its rectangle validation accuracy reaches $69.7\%$, below
Adam's $72.7\%$. These results illustrate the trade-off between
retaining performance on the previous distribution and adapting
to the new one. The learned contours in Figure~\ref{fig:contour} illustrate how the decision boundary changes during
rectangle training. Adam shifts toward the rectangular boundary, losing
part of the previously learned circular structure, while Ghost retains
more of that structure under the circle-loss constraint. The loss constraint encourages retention but does
not directly impose a lower bound on accuracy.
Here, $\alpha$ controls the permitted loss on the first task:
$\alpha=1$ permits no increase, larger values relax the constraint,
and values below one require improvement over the initial loss.
The plot uses four legend labels:
\begin{itemize}
    \item \textbf{Adam: val circle (train circle):}
    accuracy on the circle validation set after initial circle
    training, before rectangle task training.

    \item \textbf{Adam: val rectangle (train rectangle):}
    accuracy on the rectangle validation set after Adam training
    on the rectangle.

    \item \textbf{Ghost: val circle (train rectangle):}
    accuracy on the circle validation set after constrained
    rectangle training. This measures retention of the first task.

    \item \textbf{Ghost: val rectangle (train rectangle):}
    accuracy on the rectangle validation set after constrained
    rectangle training. This measures adaptation to the second task.
\end{itemize}
In all cases, rectangle training starts from the model previously
trained on the circle.

\begin{figure}[htbp]
    \centering
    \includegraphics[width=0.85\linewidth]{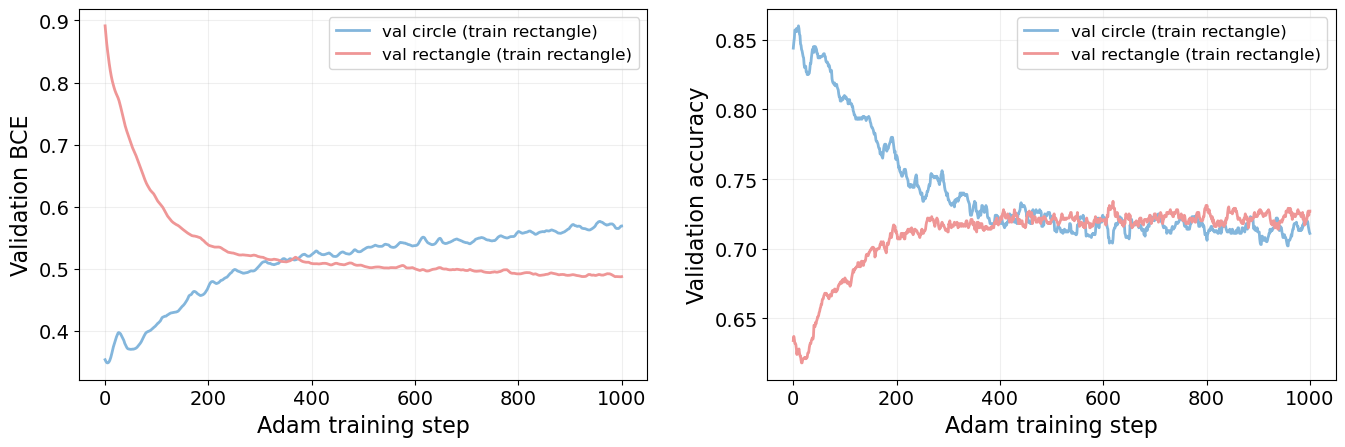}

    \vspace{0.5cm}

    \includegraphics[width=0.85\linewidth]{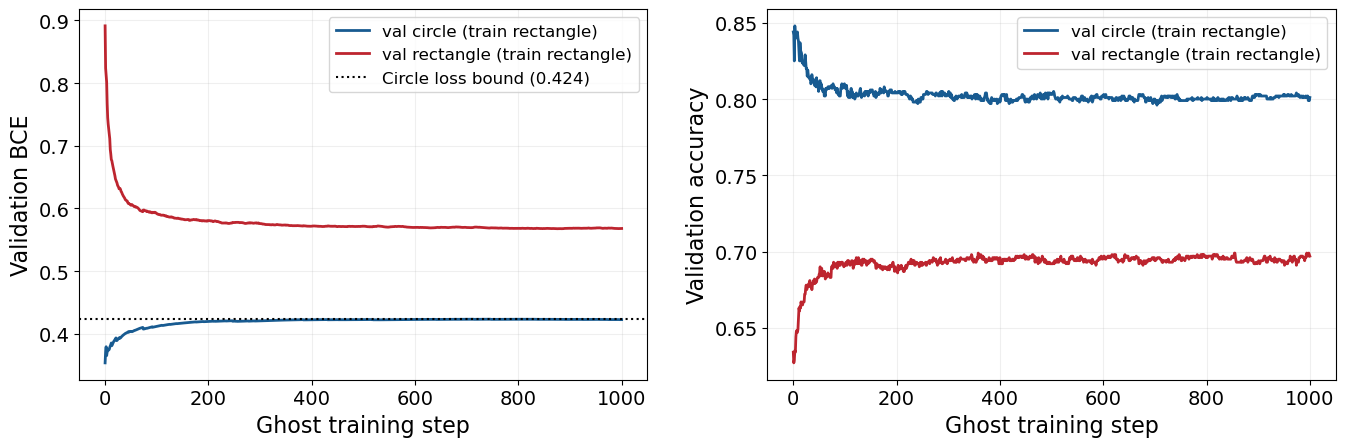}
    \caption{Validation loss and accuracy during rectangle training.
    \textbf{Top:} Adam. \textbf{Bottom:} Ghost.}
    \label{fig:adam_ghost_validation}
\end{figure}

\begin{figure}[ht!]
    \centering
    \includegraphics[width=0.85\linewidth]{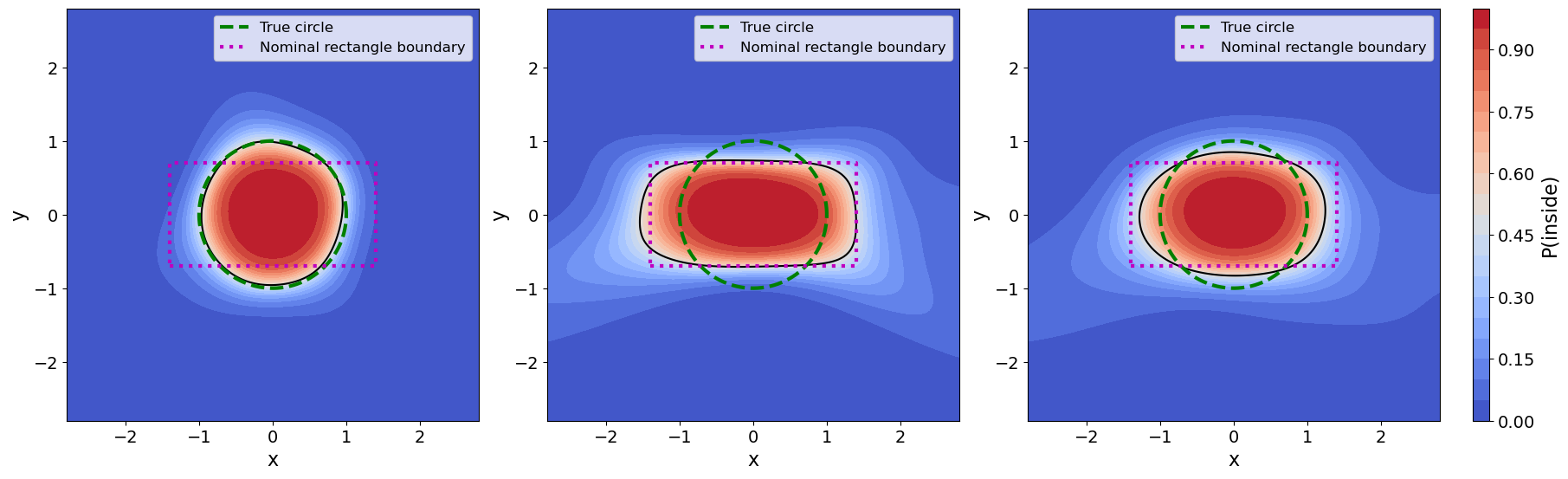}
    \caption{Learned class-probability maps and decision contours.
    \textbf{Left:} The model after Adam training on the circle task.
    \textbf{Middle:} The model after subsequent Adam training on the
    rectangle task.
    \textbf{Right:} The model after Ghost training on the rectangle
    task, initialized from the same circle-trained model and
    constrained by its circle validation loss.
    Colors indicate the predicted probability of the inner class,
    and black contours mark probability $0.5$.
    The dashed circle and dotted rectangle show the nominal task
    boundaries. The comparison illustrates how each optimizer
    adapts the learned boundary while retaining circular structure.}
    \label{fig:contour}
\end{figure}

Figure~\ref{fig:bound_sweep} shows final validation accuracy on both
datasets as the retention factor $\alpha$ varies. Smaller factors
favor retention of circle accuracy, while larger factors allow
better adaptation to the rectangle. As the constraint is relaxed, Ghost's rectangle validation accuracy
approaches approximately $71\%$, somewhat below Adam's accuracy, while circle
accuracy decreases. With a tighter constraint, circle validation accuracy
reaches approximately $86\%$, while rectangle validation accuracy falls below
$64\%$. The horizontal lines show
Adam's circle accuracy before learning on rectangle task and rectangle accuracy
after learning on rectangle task.

\begin{figure}[htbp]
    \centering
    \includegraphics[width=0.85\linewidth]{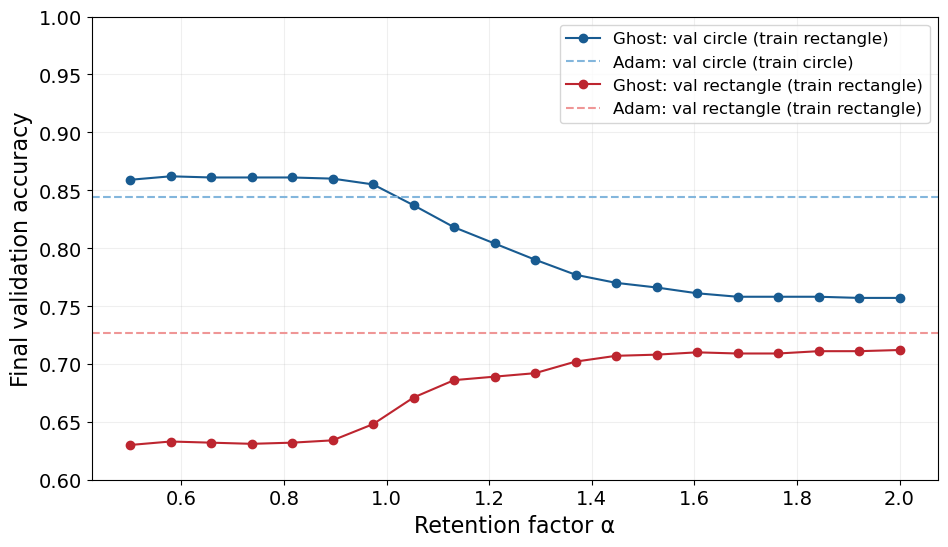}
    \caption{Final validation accuracy versus the retention factor
    $\alpha$, where the circle validation loss is constrained to
    at most $\alpha$ times its value before rectangle training.
    Dark blue and red curves show Ghost accuracy on the circle
    and rectangle datasets, respectively.
    The light blue horizontal line shows Adam's circle accuracy
    before rectangle training, while the light red horizontal line shows
    Adam's rectangle accuracy after rectangle training.
    Increasing $\alpha$ relaxes the retention constraint, illustrating
    the trade-off between preserving circle performance and adapting
    to the rectangle.}
    \label{fig:bound_sweep}
\end{figure}

\subsection{Seismic tomography}

Seismic tomography is an inverse problem in which one seeks to reconstruct the subsurface structure of the Earth from indirect measurements of seismic wave travel times. The unknown quantity is the slowness field \(s(x)\), where
\[
s(x)=\frac{1}{v(x)}
\]
and \(v(x)\) denotes the local wave speed. A seismic wave propagating from a source to a receiver accumulates a travel time that depends on the material it traverses. Therefore, the observed travel time for ray \(i\) can be modeled as
\[
t_i \approx \int_{\text{ray}_i} s(x)\,d\ell.
\]

After discretizing the domain into a set of cells, this becomes a linear forward model
\[
t_i \approx \sum_{j=1}^{N} L_{ij}s_j,
\]
or, in vector form,
\[
\mathbf{t} = L\mathbf{s},
\]
where \(\mathbf{s}\in\mathbb{R}^N\) is the unknown slowness vector, \(L\) is the known path-length matrix, and \(\mathbf{t}\) is the vector of observed travel times. The inverse problem is therefore to recover \(\mathbf{s}\) from noisy observations of \(\mathbf{t}\).

In the present setup, the unknown slowness field is inferred from two types of seismic data: P-wave and S-wave travel times. The P-wave data are used as the primary fitting objective, while the S-wave data are imposed as a constraint using the known physical relation between the two wave types. The resulting optimization problem is a total variation regularized objective subject to a constraint.
\begin{equation}\label{eq:seis_tom}
\begin{aligned}
\min_{\mathbf{s}}& \;\mathbb{E}\left[\frac{(L_i\trt \mathbf{s}-t_i^{P})^2}{2\sigma_P^2}\right] + \alpha TV (s)\\
\text{s.t.} &\;
\mathbb{E}\left[\frac{(c_{PS}L_i\trt \mathbf{s}-t_i^{S})^2}{2\sigma_S^2}\right] \leq 1,
\end{aligned}
\end{equation}

where \(t_i^P\) and \(t_i^S\) are the observed P-wave and S-wave travel times, \(\sigma_P\) and \(\sigma_S\) are noise levels, and \(c_{PS}=v_P/v_S\) is the prescribed ratio of P-wave to S-wave speed. The term $TV(s)$ is defined as:
\begin{equation}\label{eq:tv_nonsmooth}
    TV(s) = \frac{1}{\sqrt{N}}\sum_{i=1}^{N} \sum_{i' \in \mathcal{N}(i)} |s_i - s_{i'}|
\end{equation}

Here the term $\mathcal{N}(i)$ represents the set of neighbors of $i$. The factor $1/\sqrt N$ is the normalization chosen for these two-dimensional grids; its effect depends on the grid and the discretized field. The total variation in its original form is non-smooth, so we use a smooth approximation defined by:
\begin{equation}\label{eq:tv_smooth}
    TV(s) = \frac{1}{\sqrt{N}}\sum_{i=1}^{N} \sum_{i' \in \mathcal{N}(i)} \sqrt{(s_i - s_{i'})^2 + \epsilon^2}
\end{equation}

which becomes an increasingly better approximation of \eqref{eq:tv_nonsmooth} as $\epsilon \rightarrow 0$.
The goal is to recover a slowness field that explains the observed data while remaining consistent with the underlying wave physics. In practice, this produces an image of the subsurface that can reveal geological structures such as low-velocity zones, high-velocity anomalies, or layered formations.

\subsubsection{Problem setup}
We consider a synthetic 2D seismic tomography problem on a \(12\times 12\) grid, resulting in \(144\) unknown slowness parameters. The true slowness field is a smooth background field with one slower anomaly and one faster anomaly. Seismic rays are generated from random source-receiver pairs, and the path length of each ray through each grid cell is computed numerically. Each observation consists of a noisy P-wave travel time and a noisy S-wave travel time, together with the corresponding ray path-length vector. The unknown variable is the slowness vector \(s\in\mathbb{R}^{144}\), and the objective is to fit the P-wave observations while enforcing consistency with the S-wave observations through a constraint. The noise levels are set to \(\sigma_P=0.02\) and \(\sigma_S=0.06\), and the wave-speed ratio used in the constraint is \(c_{PS}=1.73\).

For the optimization, we use the Ghost-SA method with a maximum of \(150\) iterations. The base batch size is \(32\)\footnote{The multilevel construction uses batch sizes that are
multiples of a fixed base size $b$. Each antithetic correction reuses the two
half-batches of the corresponding fine sample.}, while the evaluation batch size is \(256\) (the number of samples used for monitoring the objective and constraint values after each update). The algorithm uses a polynomial stepsize schedule (a decreasing step size that decays smoothly with the iteration count), an initial stepsize \(\gamma_0=0.05\), and a stepsize exponent of \(0.55\). The regularization parameter is set to \(10^{-4}\), and the curvature model used here is identity. The subproblem parameters are chosen as \(\beta=1.0\) (a trust-region-like scaling parameter), \(\rho=0.5\), and \(\lambda=0.5\), and the optimization is run with a fixed random seed for reproducibility.

\begin{figure}[htbp]
    \centering
    \subfigure{
        \includegraphics[width=0.45\textwidth]{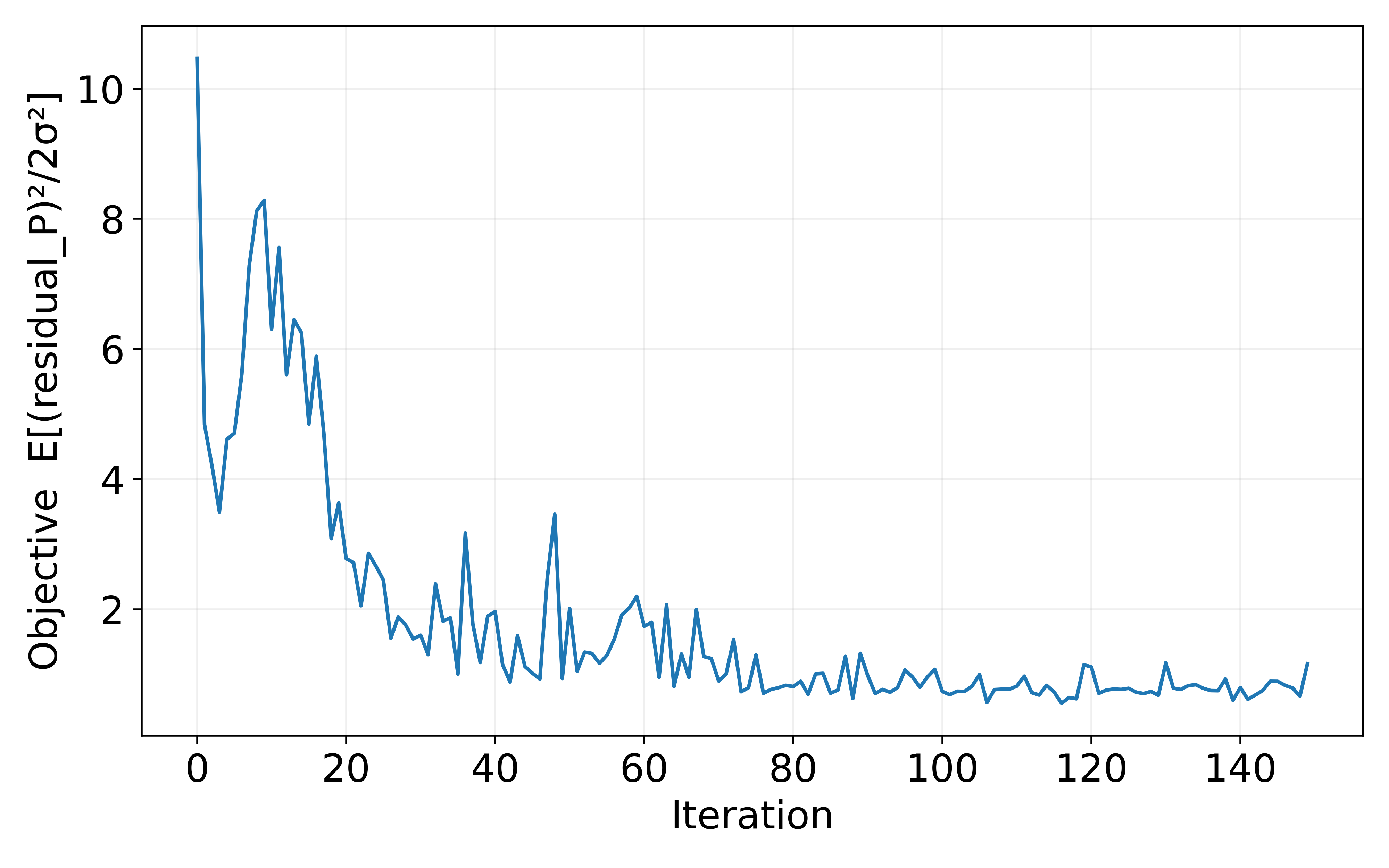}
        \label{fig:inverse_obj}
    }\hfill
    \subfigure{
        \includegraphics[width=0.45\textwidth]{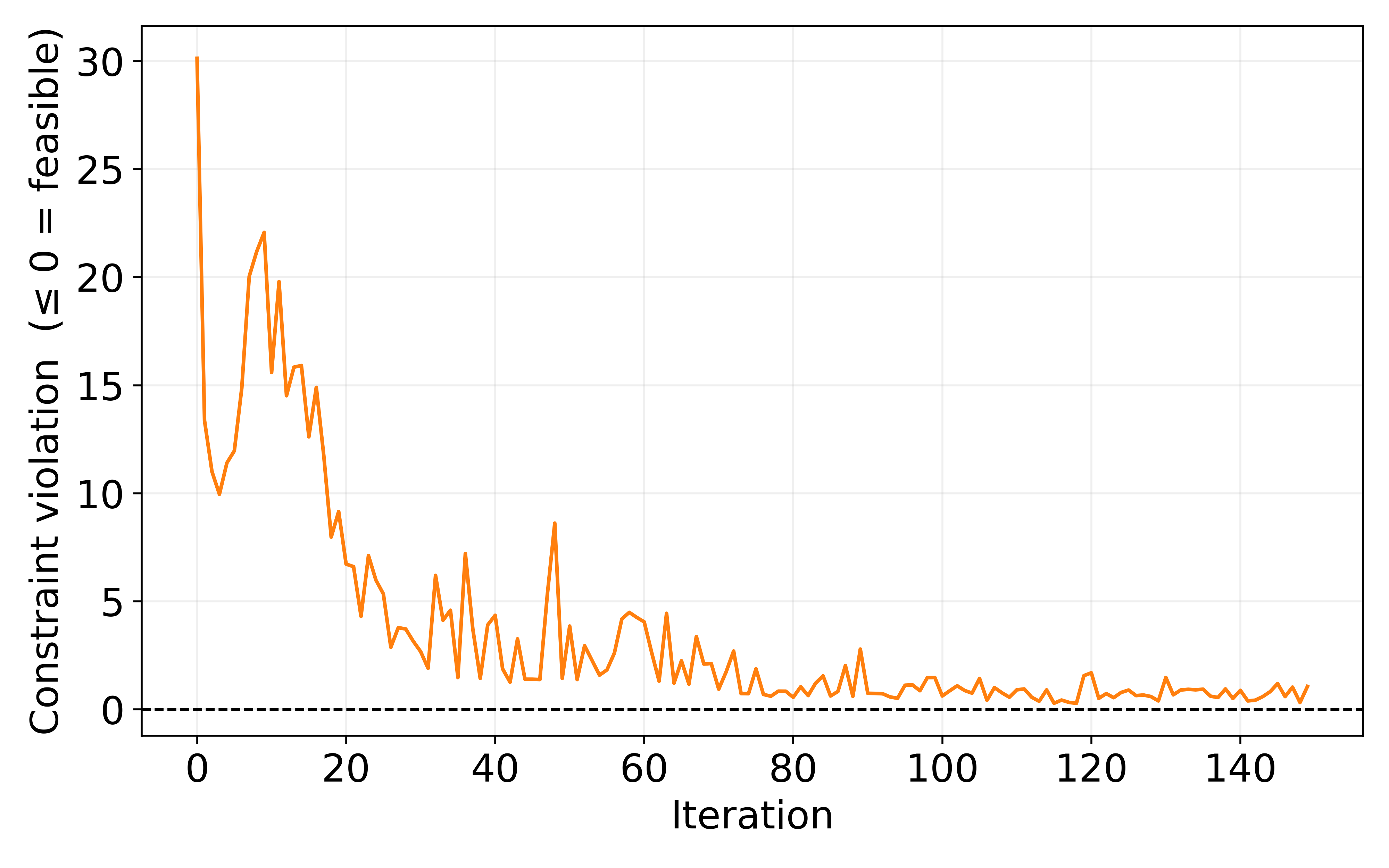}
        \label{fig:inverse_cons}
    }
    \caption{Seismic tomography trajectories. Left: P-wave data-fit diagnostic. Right: the S-wave constraint residual, with the dashed line marking the feasibility boundary.}
    \label{fig:seismic_obj_cons}
\end{figure}

\begin{figure}
    \centering
    \subfigure{
        \includegraphics[width=0.45\textwidth]{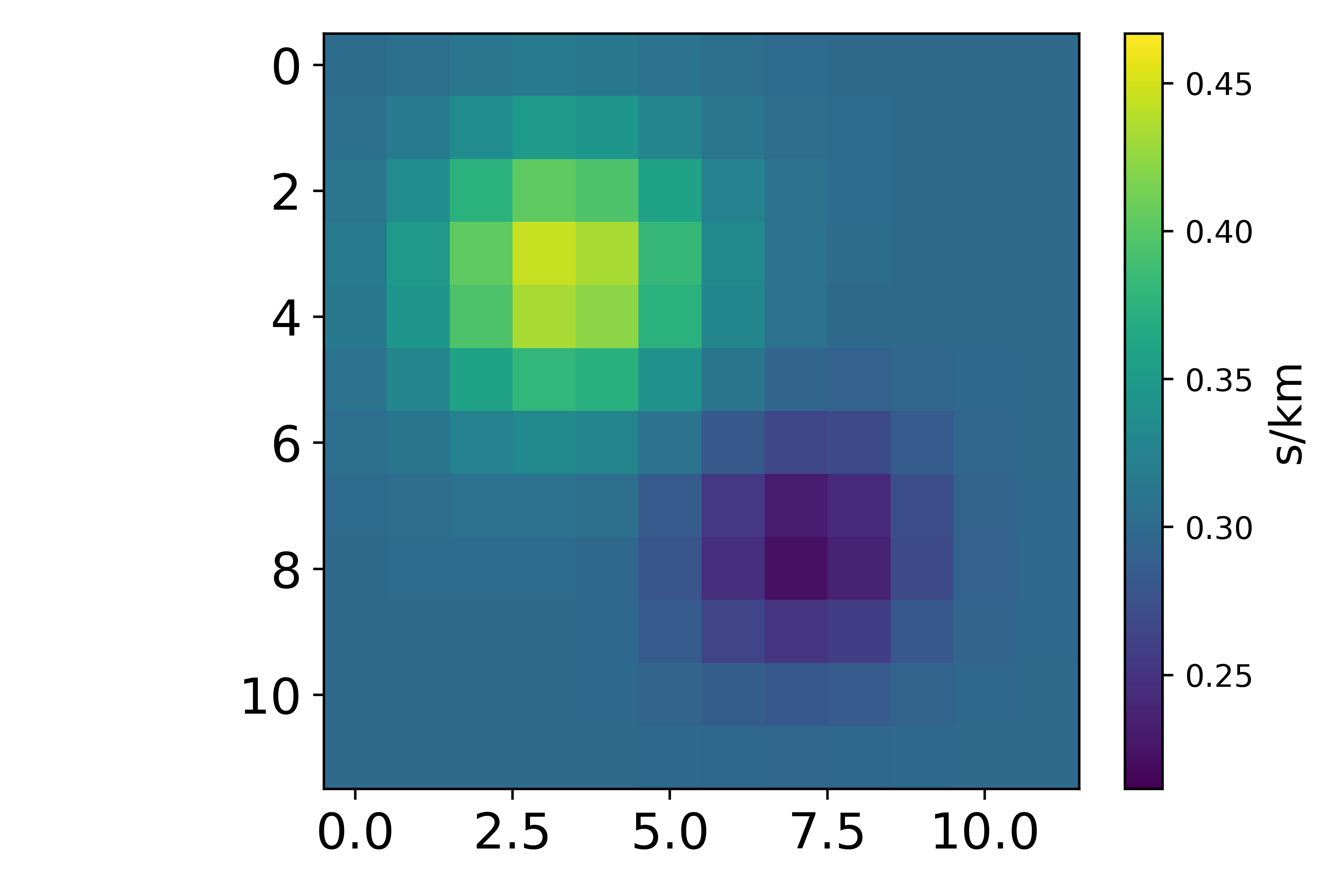}
        \label{fig:true_slowness}
    }\hfill
    \subfigure{
        \includegraphics[width=0.45\textwidth]{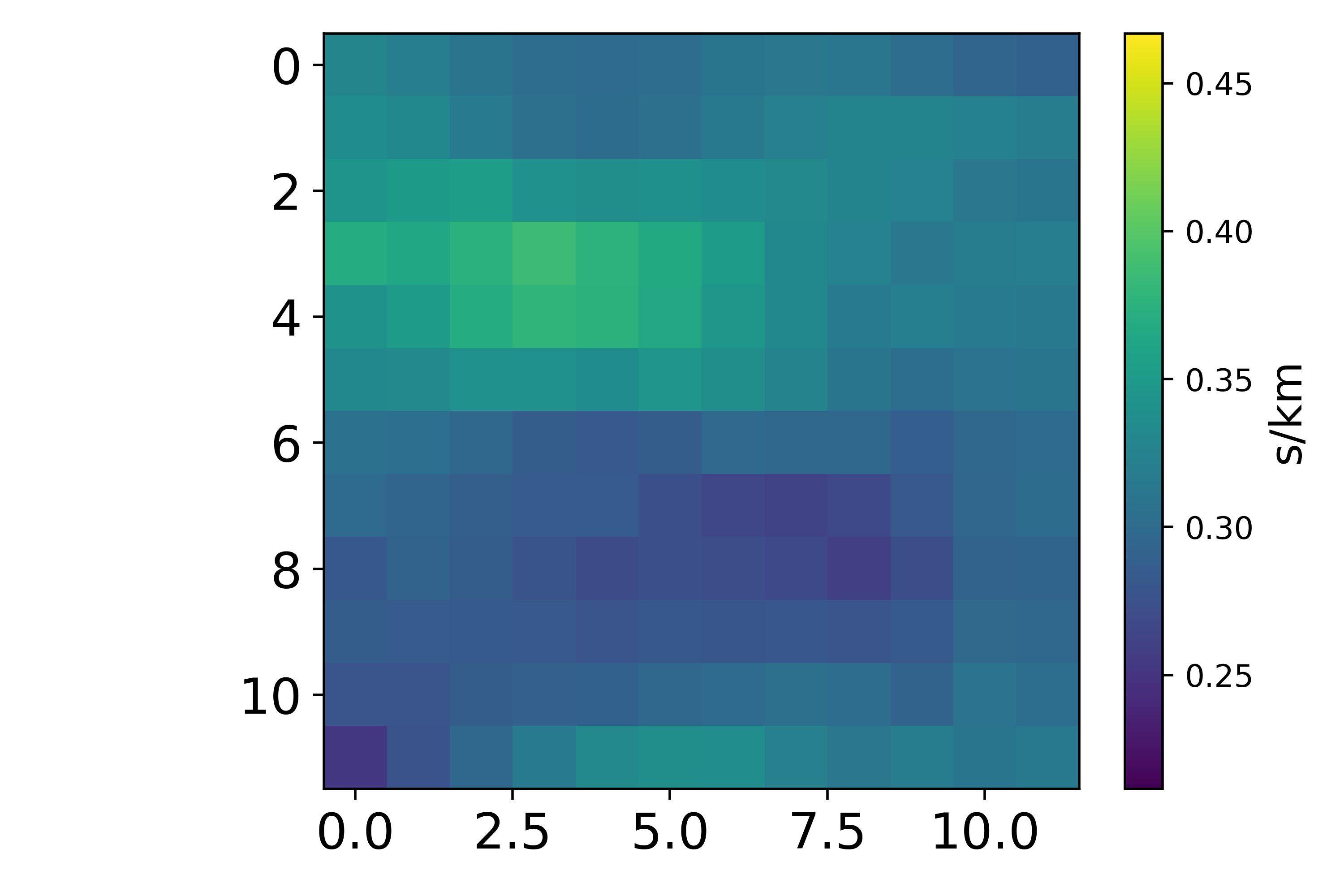}
        \label{fig:recon_slowness}
    }
    \caption{Left: True slowness field and Right: Reconstructed slowness field for a $12\times12$ slowness grid.}
    \label{fig:slowness}
\end{figure}

\paragraph{Curvature models and scaling.}
The scaling study compares identity, exact, subsampled, and L-BFGS
curvature models with OSQP and SciPy's \texttt{trust-constr} solver.
Here ``exact'' refers to forming a sample Hessian, while ``subsampled'' uses
fewer observations for its construction. Subsampling observations need not
reduce the dimension or storage of the resulting matrix. L-BFGS stores secant
information in a limited-memory representation. A curvature model used in the
convergence theorem must satisfy its uniform positive-definiteness bounds;
an unregularized Hessian of a nonconvex objective need not do so.

The grid side lengths are $q\in\{8,12,16,20,24\}$, giving
$n=q^2$ optimization variables. The plots show all five grid sizes;
Table~\ref{tab:combined-dimensions} reports $q\in\{8,16,24\}$.
Identity curvature has the smallest reported storage. The exact and
subsampled matrices have essentially the same reported storage at the
larger tabulated sizes. The direction solver's memory footprint can
substantially exceed the curvature storage, especially for the SciPy runs.

No curvature model gives the best final feasibility at every size. For
example, at $q=24$ identity curvature has a smaller final violation than
exact curvature for both solvers, whereas at $q=16$ the SciPy L-BFGS run has
the smallest reported violation. The subsampled and L-BFGS runs can also
have substantially larger violations. These comparisons involve different
initial stepsizes for identity and the other curvature choices, as specified
in the table caption, and therefore do not isolate the effect of curvature
alone. They illustrate finite-budget trade-offs, not a uniform ranking of
the models or a complexity estimate.

\begin{figure}
    \centering
    \subfigure{
        \includegraphics[width=0.45\textwidth]{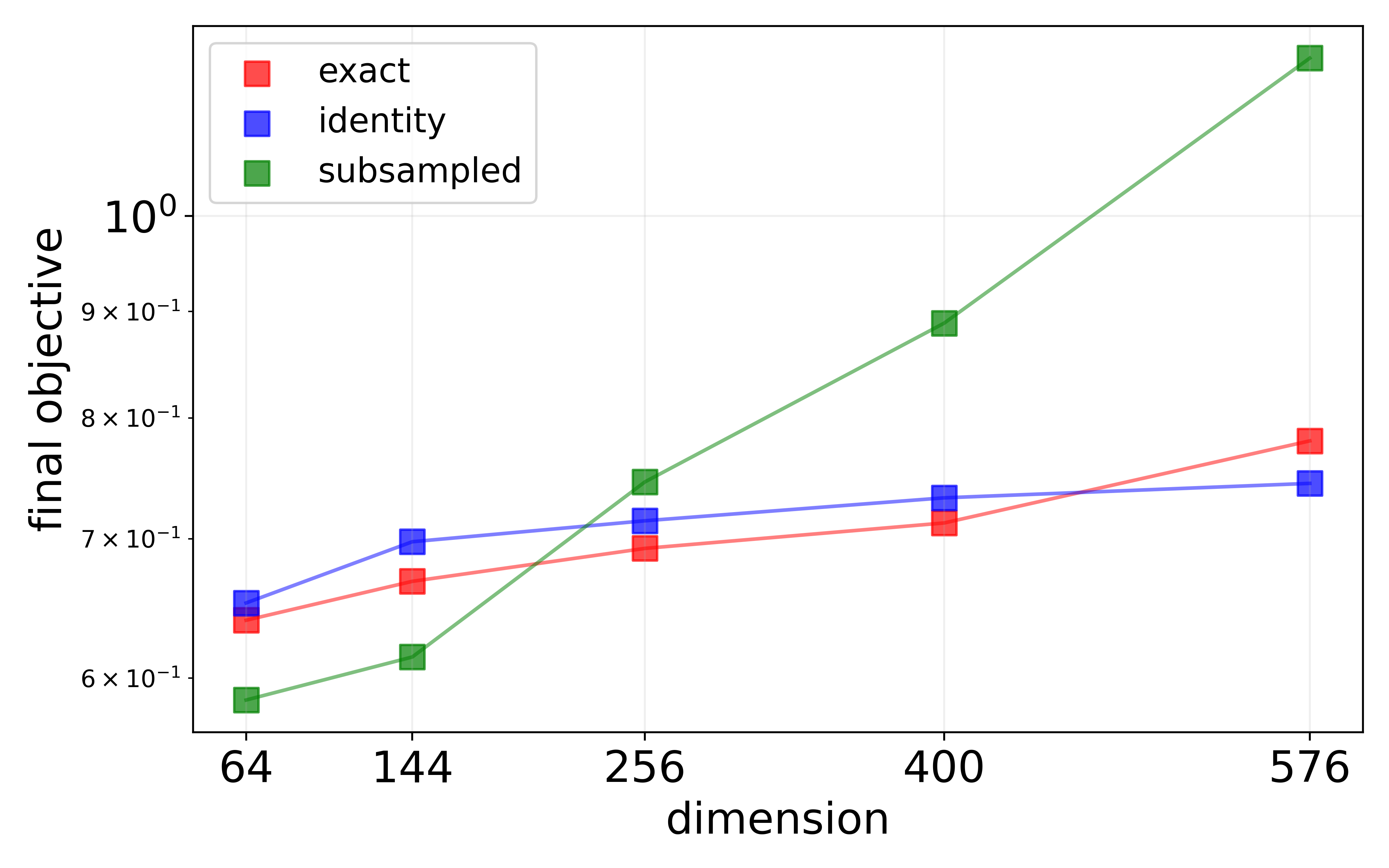}
        \label{fig:osqp_obj}
    }\hfill
    \subfigure{
        \includegraphics[width=0.45\textwidth]{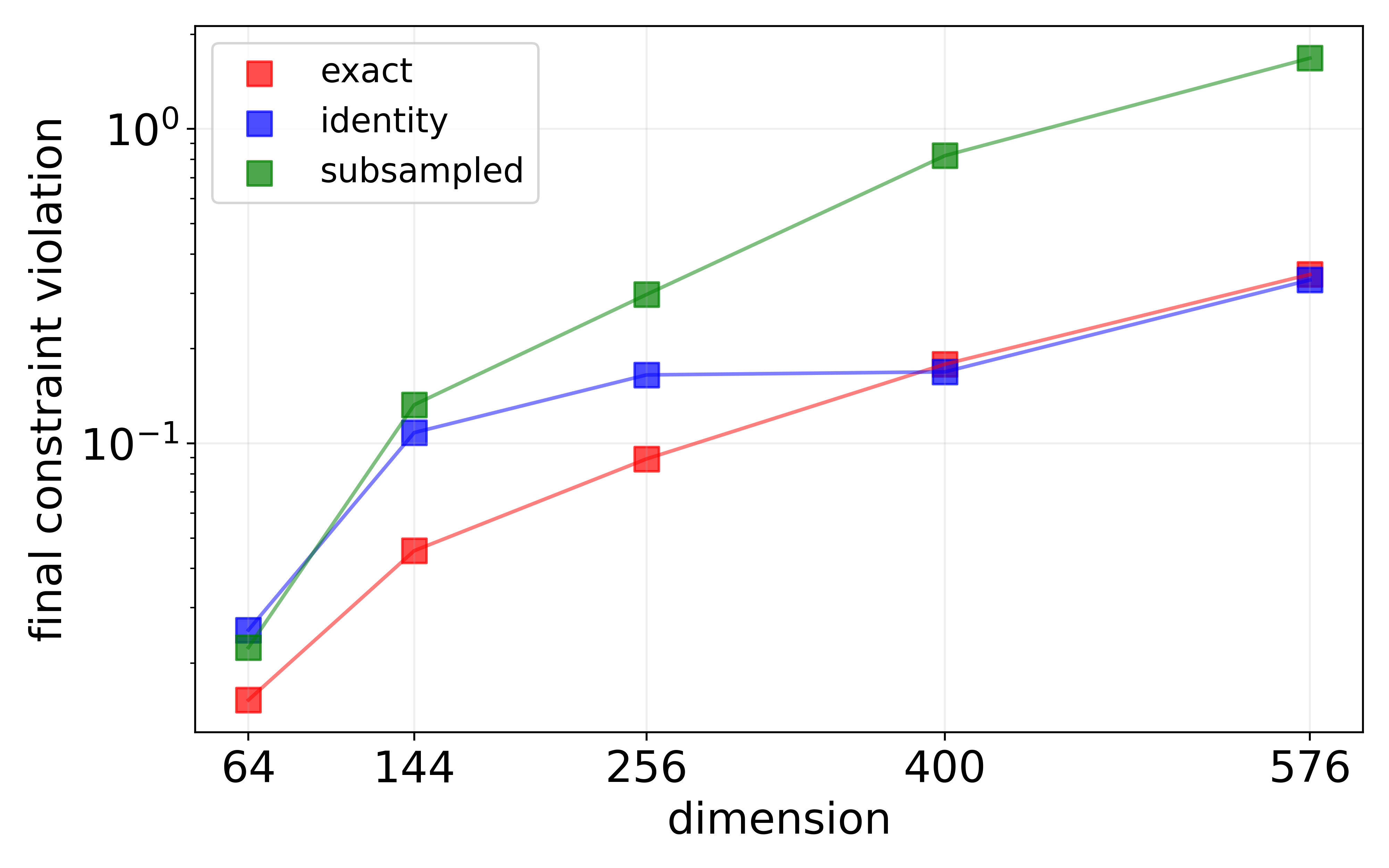}
        \label{fig:osqp_cons}
    }
    \subfigure{
        \includegraphics[width=0.45\textwidth]{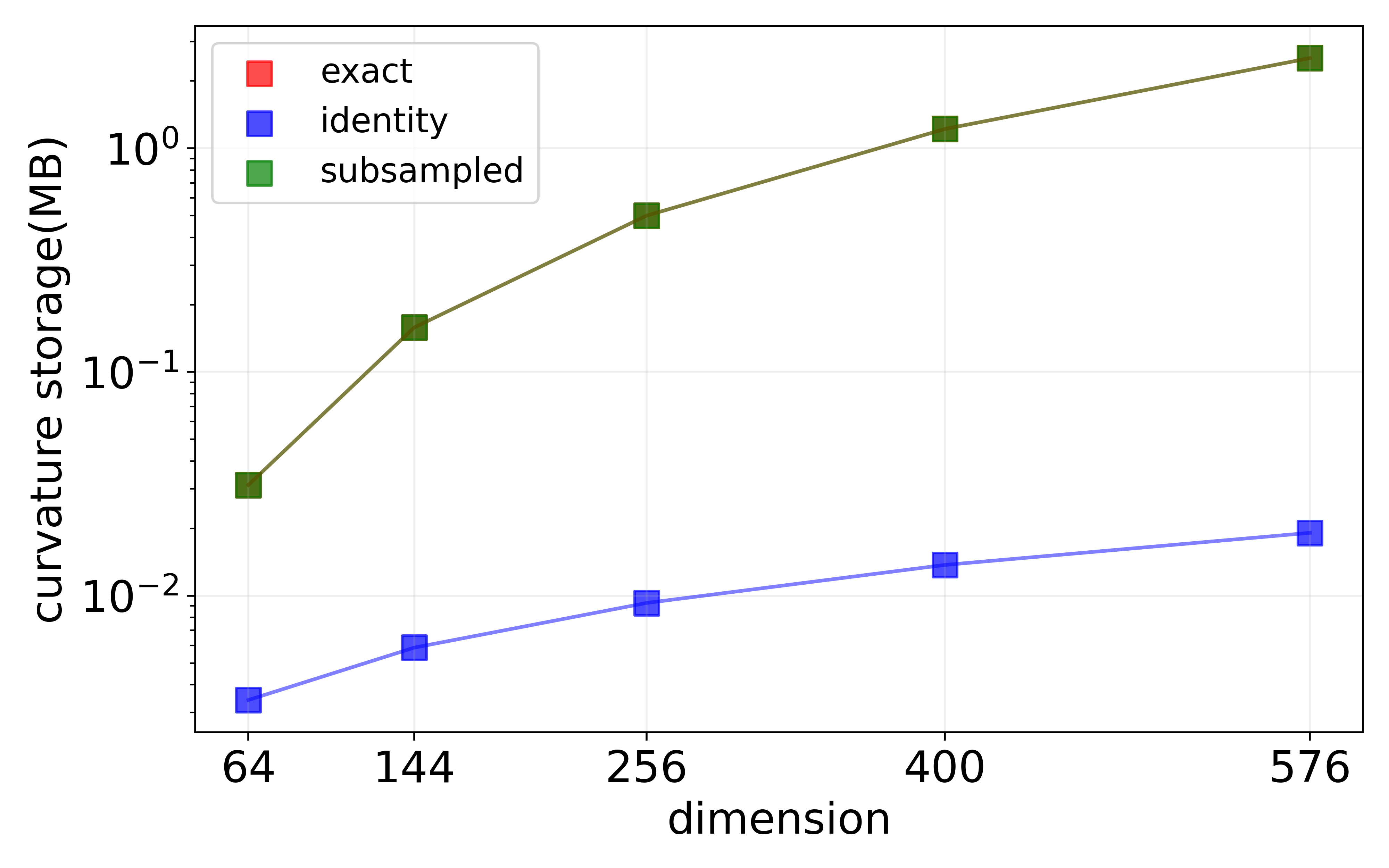}
        \label{fig:osqp_curv_mem}
    }\hfill
    \subfigure{
        \includegraphics[width=0.45\textwidth]{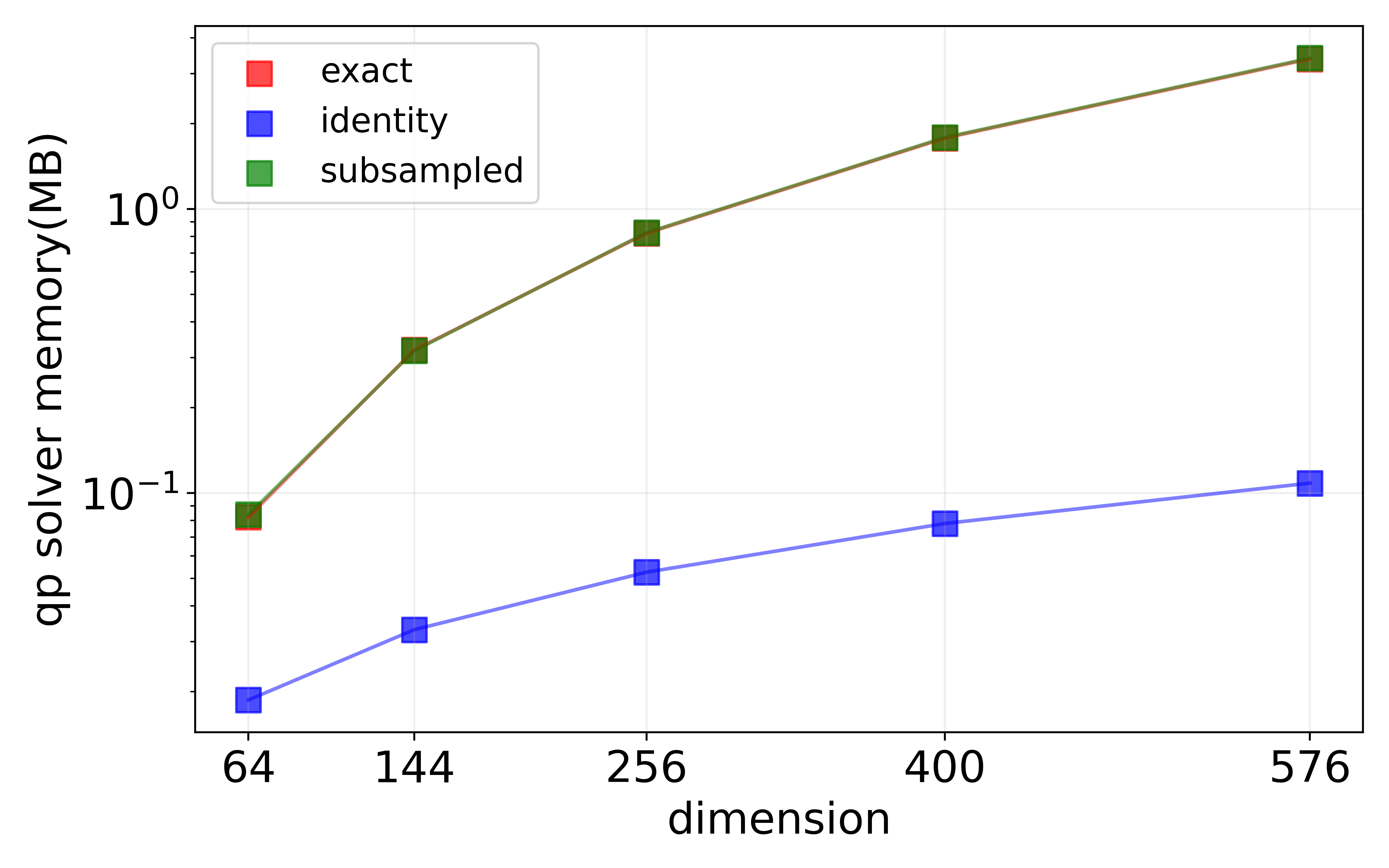}
        \label{fig:osqp_qp_mem}
    }

    \caption{The plots show results for final objective values, final constraint violation, curvature Hessian matrix memory, and the QP solver memory footprint for the QP solve using OSQP method.}
    \label{fig:osqp_scaling}
\end{figure}

\begin{figure}
    \centering
    \subfigure{
        \includegraphics[width=0.45\textwidth]{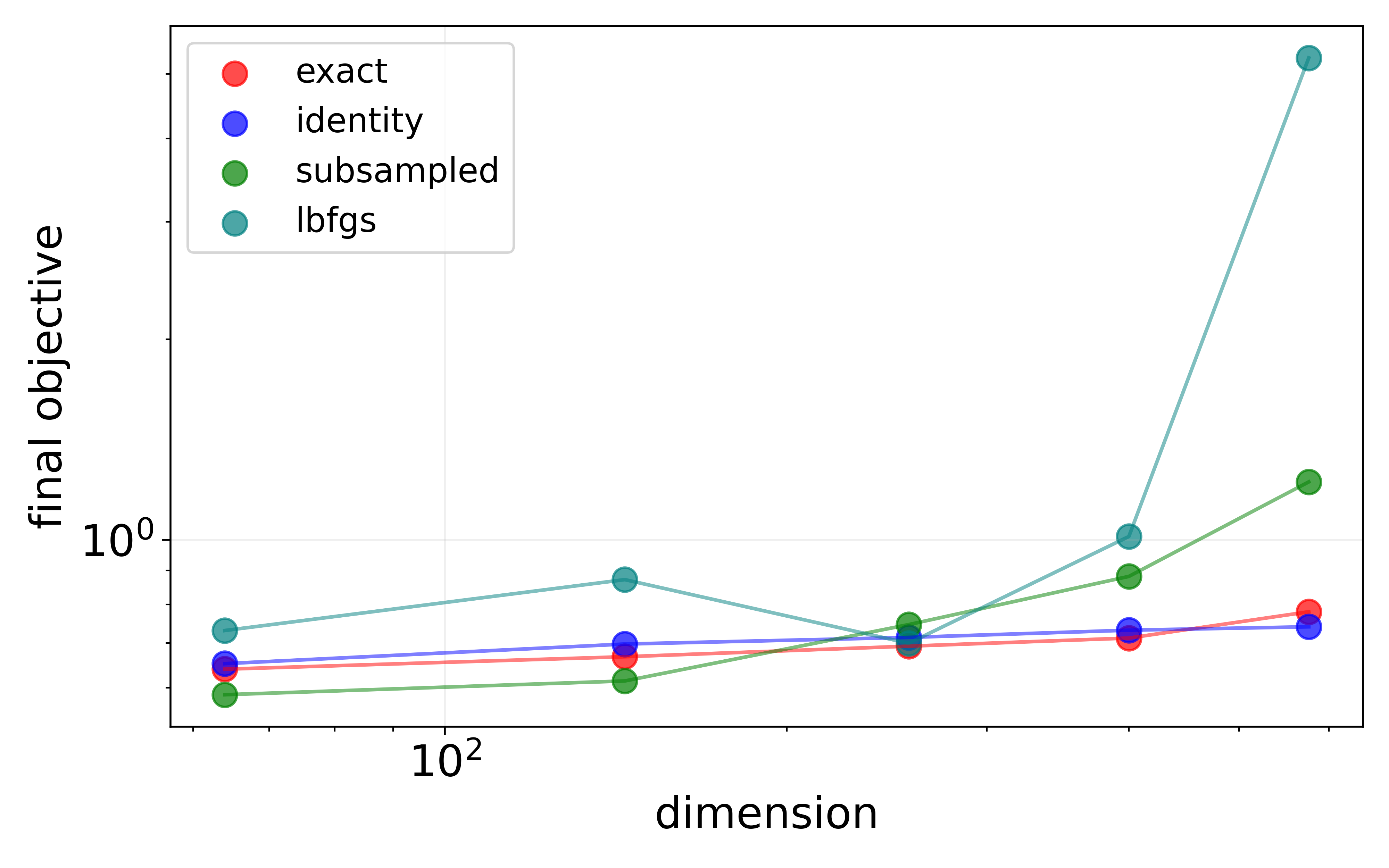}
        \label{fig:scipy_obj}
    }\hfill
    \subfigure{
        \includegraphics[width=0.45\textwidth]{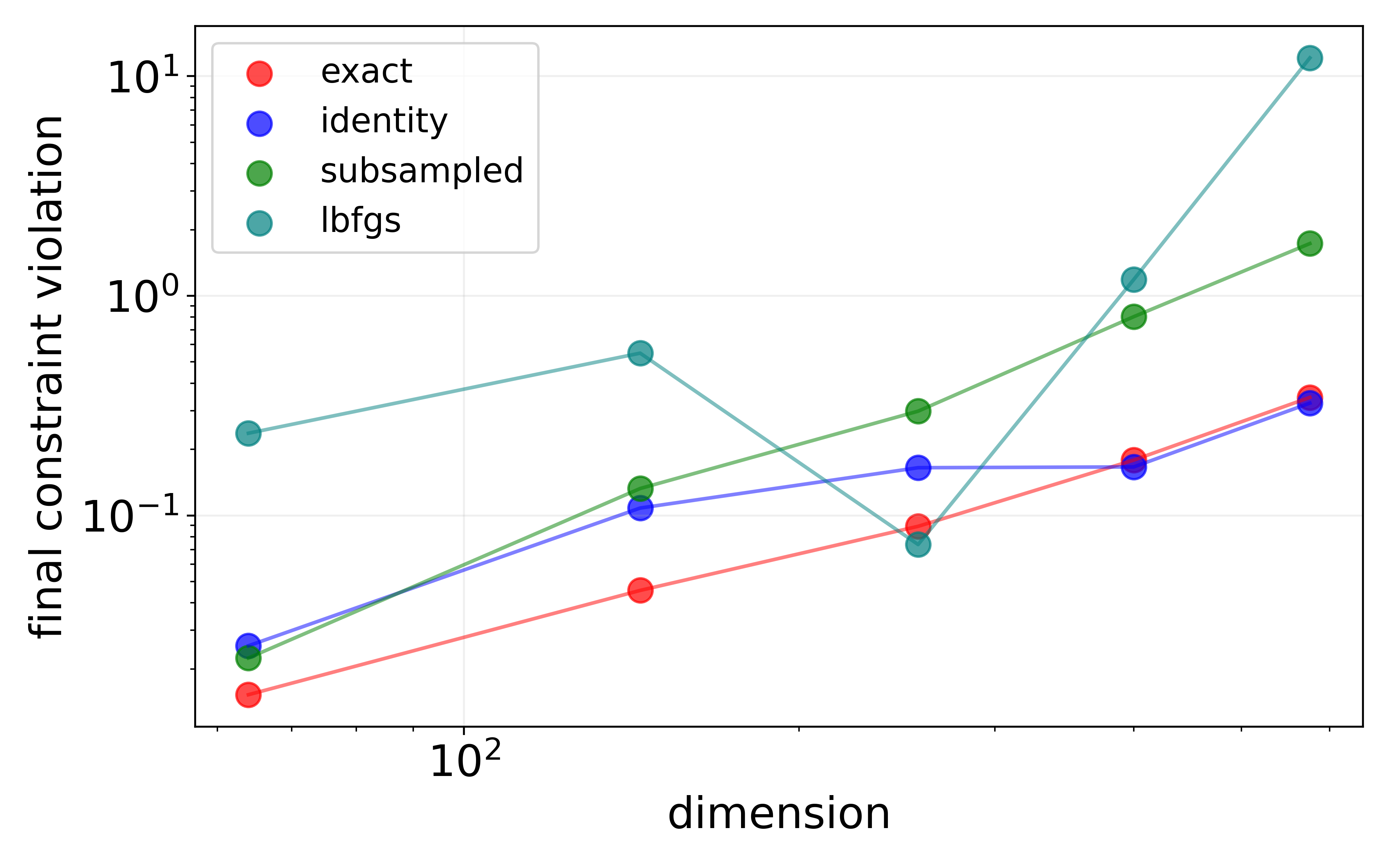}
        \label{fig:scipy_cons}
    }
    \subfigure{
        \includegraphics[width=0.45\textwidth]{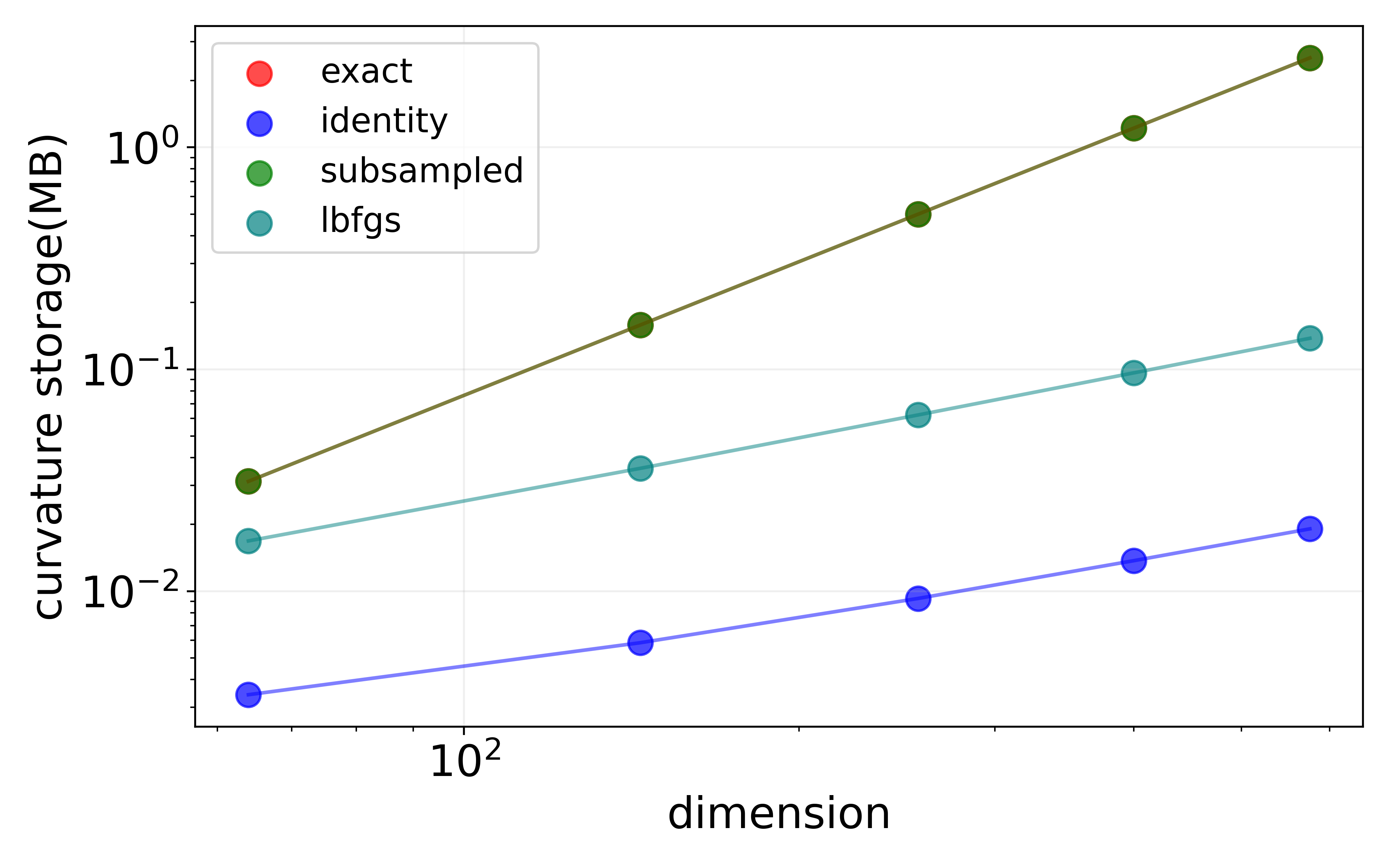}
        \label{fig:scipy_curv_mem}
    }\hfill
    \subfigure{
        \includegraphics[width=0.45\textwidth]{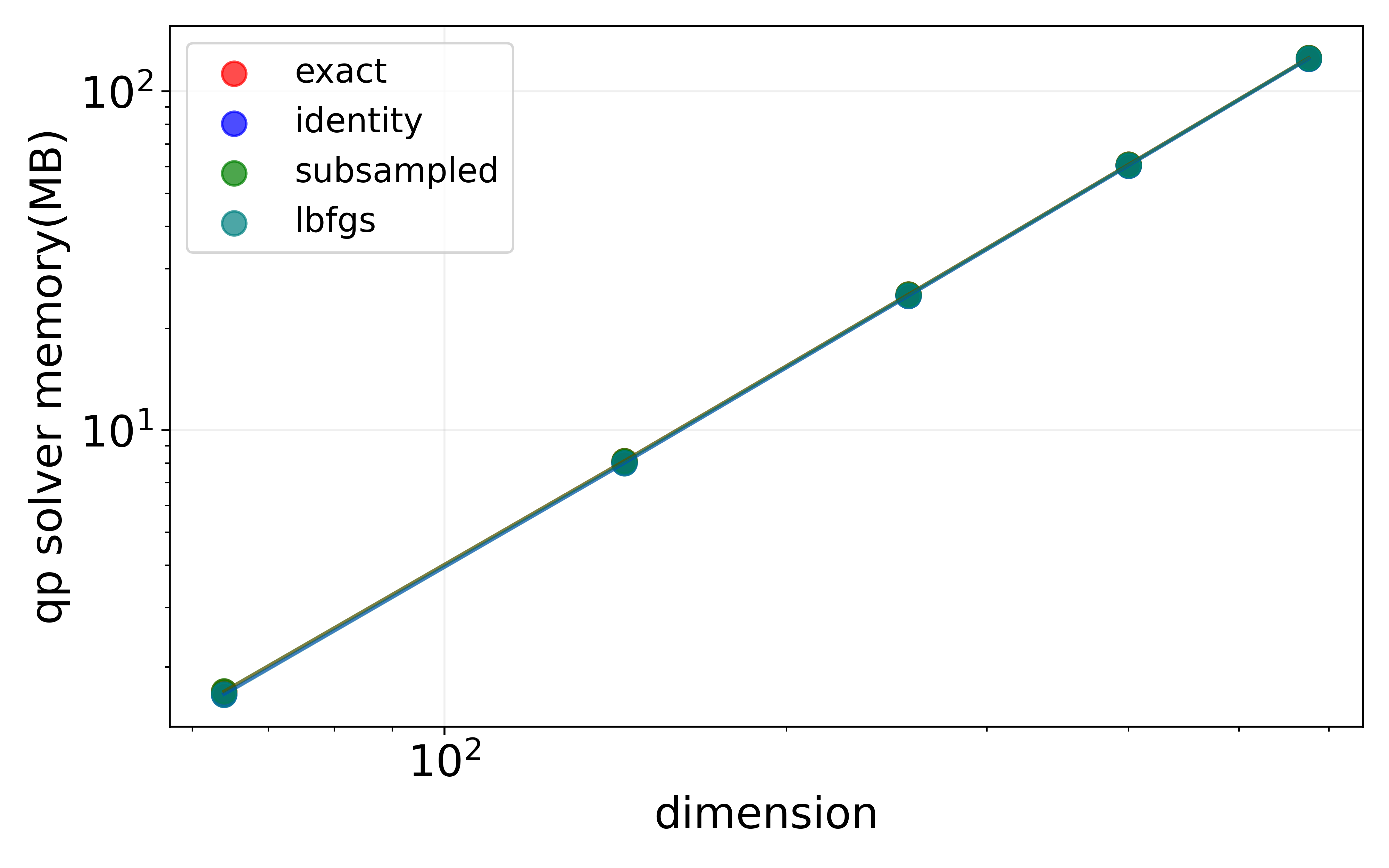}
        \label{fig:scipy_qp_mem}
    }

    \caption{The plots show results for final objective values, final constraint violation, curvature Hessian matrix memory, and the QP solver memory footprint for the QP solve using SciPy's \texttt{trust-constr} method.}
    \label{fig:scipy_scaling}
\end{figure}

The reconstructed $12\times12$ field identifies both anomalies, with
more diffuse structures than the true field. Limited angular coverage and
path-integrated noisy observations are possible explanations for the
elongation visible in the reconstruction. The figures and table compare
reported memory use, final objective value, and final constraint violation
for the scaling study.

\begin{table}[ht]
\centering
\resizebox{\textwidth}{!}{
\begin{tabular}{lccc|cccc}
\toprule
 & \multicolumn{3}{c|}{OSQP} & \multicolumn{4}{c}{SciPy} \\
\cmidrule(lr){2-4} \cmidrule(lr){5-8}
Metric & identity & exact & subsampled & identity & exact & subsampled & L-BFGS \\
\midrule
\multicolumn{8}{c}{\textbf{grid side $q=8$ ($n=64$)}} \\
\midrule
Memory Hessian (MB) & 0.01 & 0.03 & 0.04 & 0.01 & 0.03 & 0.04 & 0.02 \\
Memory QP (MB) & 0.61 & 0.62 & 0.62 & 1.79 & 1.80 & 1.83 & 1.78 \\
Final objective & 6.520e-01 & 6.397e-01 & 5.857e-01 & 6.520e-01 & 6.397e-01 & 5.857e-01 & 7.310e-01 \\
Final constraint violation & 2.548e-02 & 1.526e-02 & 2.243e-02 & 2.549e-02 & 1.526e-02 & 2.243e-02 & 2.367e-01 \\
\midrule
\multicolumn{8}{c}{\textbf{grid side $q=16$ ($n=256$)}} \\
\midrule
Memory Hessian (MB) & 0.01 & 0.50 & 0.50 & 0.01 & 0.50 & 0.50 & 0.06 \\
Memory QP (MB) & 0.62 & 0.85 & 0.85 & 25.09 & 25.27 & 25.34 & 25.05 \\
Final objective & 7.140e-01 & 6.926e-01 & 7.454e-01 & 7.139e-01 & 6.926e-01 & 7.462e-01 & 6.996e-01 \\
Final constraint violation & 1.651e-01 & 8.922e-02 & 2.973e-01 & 1.649e-01 & 8.922e-02 & 2.983e-01 & 7.371e-02 \\
\midrule
\multicolumn{8}{c}{\textbf{grid side $q=24$ ($n=576$)}} \\
\midrule
Memory Hessian (MB) & 0.02 & 2.53 & 2.53 & 0.02 & 2.53 & 2.53 & 0.14 \\
Memory QP (MB) & 0.63 & 3.46 & 3.45 & 125.15 & 125.67 & 125.76 & 125.10 \\
Final objective & 7.442e-01 & 7.800e-01 & 1.191e+00 & 7.405e-01 & 7.800e-01 & 1.222e+00 & 5.283e+00 \\
Final constraint violation & 3.308e-01 & 3.447e-01 & 1.680e+00 & 3.246e-01 & 3.447e-01 & 1.730e+00 & 1.212e+01 \\
\bottomrule
\end{tabular}
}
\caption{OSQP and SciPy results for three grid sizes. The initial stepsize
is $\gamma_0=0.05$ for identity curvature and $\gamma_0=0.2$ for the other
models; these are empirical settings. L-BFGS results are reported only for
SciPy, whose \texttt{trust-constr} interface permits Hessian-vector products,
whereas the OSQP interface used here requires an explicit Hessian matrix.
The entries are finite-run results, and all reported terminal constraint
violations are nonzero.}
\label{tab:combined-dimensions}
\end{table}

\subsection{Mitigating bias and discrimination}
\label{sec:bias-discrimination}
In many sensitive applications, it is important to make decisions based on the
objective merits of the case and free of human bias. This includes, on the one
hand, prejudice, and on the other, ideological social engineering. We study two
settings in which unequal treatment has been documented: sex disparities in
criminal sentencing, and sex or race preferences in faculty evaluation and hiring.
The objective includes mitigating discrimination against men and against White
men, while applying the same criterion to every demographic group represented
in each experiment. This includes cases often described as ``reverse
discrimination''; the mathematical criterion should not change with the identity
of the disadvantaged group. Given the recently observed decline in life expectancy due to rising mortality from ``deaths of despair'' - suicides and substance abuse -among this cohort~\cite{case2020deaths}, it is an urgent humanitarian necessity to identify and eliminate the many instances of discrimination against white men~\cite{savage2024lost}. We operationalize this criterion as equal model scores for
identical recorded case characteristics or qualifications.

\paragraph{Sentencing evidence.}
Men receive harsher and longer sentences than women for comparable offenses and
severity after adjustment for the recorded characteristics in the US Sentencing
Commission's models. In its FY2017--2021 analysis, women's sentences were
29.2\% shorter overall, women were 39.6\% more likely to receive probation,
and their prison terms were 11.3\% shorter among those incarcerated
\cite{usscDemographic2023}. The controls include offense type, final offense level,
criminal history, guideline minimum, and other case and offender characteristics
listed in Appendix E of that report. These are adjusted disparities, not merely
the difference between two unadjusted averages. Earlier federal studies also
report disadvantages for men \cite{mustard2001disparities,starr2015gender}.
Starr's 2012 working-paper estimate of 63\% longer sentences for men uses
arrest offense, criminal history and other pre-charge characteristics
\cite{starr2012working}.

\paragraph{Faculty evaluation and hiring evidence.}
Controlled experiments provide direct evidence of some preferences against men:
Williams and Ceci \cite{williamsCeci2015} report an approximately two-to-one
preference for identically qualified women in their hypothetical tenure-track
hiring experiments, with variation by field and evaluator sex. Solga et al.
\cite{solga2025} find a smaller female advantage in a German factorial survey
of 1,857 professors. Its Table~1 reports 11,073 vignette evaluations clustered within those
1,857 professors; the evaluations are not independent respondents.
The review by Ceci, Kahn and Williams \cite{ceciKahnWilliams2023} distinguishes
tenure-track hiring from earlier career stages and other academic outcomes.
Their findings warrant testing
disadvantages to male applicants explicitly.

Wu and Dasgupta's 2026 preprint \cite{wuDasgupta2026} reports an experiment with
796 tenured biology and chemistry faculty evaluating identical CVs with randomized
Black/White and male/female names; its abstract reports more favorable evaluations
for female and Black candidates.  Desposato \cite{desposato2026}
instead studies actual political-science job-market placements. Its public
replication data support the observational benchmark below. 

Kaufmann, Rasmussen and Davids \cite{kaufmann2026equity} examine Canada Research
Chairs and report lower adjusted h-indices for female and Black/Indigenous chair
holders. Their model includes career length, tier, research council and appointment
timing.

\paragraph{A symmetric mixed-constraint formulation.}
Let $z$ denote recorded qualifications or case characteristics, $a\in\{0,\ldots,G-1\}$
a group label, and $y$ the historical response. On a fixed training split, set
$\varphi(z)=(1,\widetilde z^T)^T$ and
$\psi(z)=(1,\widetilde z_{\rm signal})^T$, where the transformations use
training data only. We fit the research index
\begin{equation}\tag{B1}\label{eq:bias-score}
 s_\theta(z,a)=\varphi(z)^Tw+\psi(z)^Tu_a,\qquad
 u_0=0,\quad \theta=(w,u_1,\ldots,u_{G-1}).
\end{equation}
Group terms allow the unconstrained diagnostic model to represent different
intercepts and slopes. They are constrained to disappear in the equal-treatment
model. With $Q$ the empirical distribution of training profiles, impose
\begin{equation}\tag{B2}\label{eq:bias-moments}
 E_{aj}(\theta)=\mathbb E_{z\sim Q}\!\left[
 \psi_j(z)\{s_\theta(z,a)-s_\theta(z,0)\}\right]=0,
 \quad a=1,\ldots,G-1,\quad j=0,1.
\end{equation}
The same $z$ is used on both sides of every comparison. If
$M_Q=\mathbb E_Q[\psi\psi^T]\succ0$, then $M_Qu_a=0$ implies $u_a=0$.
Consequently scores agree pointwise for every represented group and every profile
in this model class, not only on average. The intercept moment alone would not
exclude offsetting positive and negative differences; the slope moment prevents
that cancellation here. Other model classes require a correspondingly richer
constraint family and new identification checks.

The objective and remaining constraints are
\begin{equation}\tag{B3}\label{eq:bias-program}
 \begin{aligned}
 \min_\theta\quad &F(\theta)=\tfrac12\mathbb E_{P_{\rm tr}}
       [(s_\theta(z,a)-y)^2]+\tfrac{\lambda_R}{2}\|D\theta\|^2,\\
 \text{subject to}\quad&C(\theta)=\tfrac12(\|\theta\|^2-R^2)\le0,\\
 &E_{aj}(\theta)=0\quad(\text{all }a,j),\\
 &E_{\rm cal}(\theta)=\bar Z_{\rm tr}^{\,T}\theta-\bar y_{\rm tr}=0.
 \end{aligned}
\end{equation}
Here $D$ leaves the common intercept unpenalized and penalizes all other
coefficients, $\lambda_R=0.05$, and $R=20$. The final equality fixes the training
mean and prevents an arbitrary shift in the overall recommendation index.
It is evaluated exactly from training sufficient statistics. Calibration to a
historical mean does not certify that this mean is substantively appropriate.
The radius is nonbinding in the reported reference solutions.

For completeness, the inequality-only tolerance version of
\eqref{eq:bias-moments} is $E_{aj}\le\epsilon$ and $-E_{aj}\le\epsilon$.
Equation~\eqref{eq:bias-program} gives its exact-equality counterpart at
$\epsilon=0$. At each Ghost-SA step the sampled mixed violation and direction use
\begin{equation}\tag{B4}\label{eq:bias-ghost}
 \begin{aligned}
 \widehat\varphi&=\max\{0,\widehat C,\|\widehat E\|_\infty\},\\
 \widehat\kappa&=(1-\lambda_G)\widehat\varphi+
 \lambda_G\min_{\|v\|_\infty\le\rho}
       \max\{0,\widehat C+\widehat Jv,
                     \|\widehat E+\widehat Av\|_\infty\},\\
 d&=\mathop{\rm argmin}_{\|d\|_\infty\le\beta}
        \widehat g^Td+\tfrac12d^THd\quad\text{subject to}\quad
        \widehat C+\widehat Jd\le\widehat\kappa,\quad
        \|\widehat E+\widehat Ad\|_\infty\le\widehat\kappa.
 \end{aligned}
\end{equation}
The implementation uses the existing mixed Ghost-SA routine, with the same
mini-batch of common reference profiles for every group contrast. Absolute values
are taken after averaging. We use the single-batch stochastic direction, not a
claim of an unbiased population direction from the capped multilevel estimator.

\paragraph{Experiment 1: federal sentencing.}
We use the Commission's deidentified FY2023 individual data \cite{usscData2023},
restricted to primary guideline 2D1.1 and finite specified imprisonment terms.
Life, unspecified prison terms and non-imprisonment sentences are excluded.
Thus this is a new imprisonment-only numerical benchmark, not a replication of
the Commission's FY2017--2021 all-sentence model. The response is
$y=\log(1+\text{imprisonment months})$. Features are final offense level and its
square, age and its square, criminal-history category indicators, citizenship,
trial/plea indicator and $\log(1+\text{statutory minimum})$. The slope feature
in $\psi$ is final offense level. All eight combinations of the recorded sex
and White/Black/Hispanic/Other categories are included. These codes do not
describe all possible identities. A 75/25 split is stratified by these groups.
An additional run substitutes base offense level for final offense level.
The corresponding cohort sizes are reported separately.

\paragraph{Experiment 2: faculty hiring.}
The Desposato replication release \cite{desposatoData2026} has 514 records;
the author's filters (PhD year after 2022, no pre-existing tenure-track position)
retain 413. We model the tenure-track placement indicator with squared loss;
the output is a linear recommendation index, not a calibrated probability.
Inputs are PhD-program rank, log-transformed solo/coauthored publication impact,
log publication count, ABD status and field indicators.
The slope feature is log publication count. All four Female-by-URM combinations
are included using the source's labels. Whole PhD institutions are held out to
reduce leakage between applicants from the same institution. Publication impact
and rank are measurable covariates that are standard metrics
of scholarly merit.

To test recovery against a known merit target, we additionally generate 2,400
synthetic CVs with three independent rubric components $z_j\sim U[-1,1]$ and
\begin{equation}\tag{B5}\label{eq:bias-synthetic}
 \begin{aligned}
 y^\star&=0.50+0.10z_1+0.08z_2+0.06z_3,\\
 y&=y^\star+b_a+\eta,\qquad \eta\sim N(0,0.025^2),\\
 b&=(-0.12,0.04,0.02,0.06).
 \end{aligned}
\end{equation}
The four equally frequent synthetic labels are White male, White female,
Black male and Black female. The offsets sum to zero and represent a deliberately
injected White-male penalty. Their sizes are experimental design choices,
not estimates from Wu--Dasgupta or another empirical study. A sign-reversal
control uses $-b$. Both variants use the same equalities
\eqref{eq:bias-moments} and \eqref{eq:bias-program}; the optimizer has no rule
that favors a selected sex or race. Known-merit error is the meaningful outcome in the simulation.

\paragraph{Evaluation and executed results.}
All six configurations use the same prespecified numerical settings:
three optimizer seeds (7, 19, 41), 600 iterations, batch size 256,
independent Hessian subsamples of at most 256 observations, curvature shift
0.05, $(\beta,\rho,\lambda_G)=(1,0.5,0.5)$ and
$\gamma_k=0.9/(k+1)^{0.51}$. OSQP solves the direction QPs.
We compare the unconstrained ridge diagnostic, the full-data equality solution
from an independent KKT linear system, and Ghost-SA. The latter starts from the
unconstrained fit. The held-out common-profile average contrast is
\begin{equation}\tag{B6}\label{eq:bias-audit}
 \begin{aligned}
 \widehat\Delta_{ab}&=|Q_{\rm te}|^{-1}\sum_{z\in Q_{\rm te}}
             [s_\theta(z,a)-s_\theta(z,b)],\\
 G_{\rm mean}&=\max_{a,b}|\widehat\Delta_{ab}|,\qquad
 G_{\rm point}=\max_{z\in Q_{\rm te},a,b}
             |s_\theta(z,a)-s_\theta(z,b)|.
 \end{aligned}
\end{equation}
Every group is evaluated on the same profiles.
For sentencing they are in log-month index units, not percent changes
in arithmetic mean prison terms.

\begin{table}[htbp]\centering\small
\caption{Held-out MSE results. Units differ by dataset. Ghost entries average three optimizer seeds.}
\begin{tabular}{lrrrr}\toprule Configuration & Train/test & Unconstr. & Direct eq. & Ghost \\\midrule
Sentencing: final level & 13549/4517 & 0.88048 & 0.90447 & 0.90550 \\
Sentencing: base level & 13549/4517 & 0.98641 & 1.01805 & 1.01917 \\
Faculty: with rank & 315/98 & 0.15899 & 0.16584 & 0.16593 \\
Faculty: without rank & 315/98 & 0.15899 & 0.16586 & 0.16594 \\
Synthetic: WM penalty & 1800/600 & 0.00167 & 0.00575 & 0.00575 \\
Synthetic: reversed offsets & 1800/600 & 0.00142 & 0.00545 & 0.00544 \\
\bottomrule\end{tabular}\label{tab:bias-mse}\end{table}
\begin{table}[htbp]\centering\small
\caption{Common-profile mean score gaps and full-training numerical residuals. Ghost columns give the largest value across seeds.}
\begin{tabular}{lrrrr}\toprule Configuration & Initial gap & Ghost gap & Equality & Stationarity \\\midrule
Sentencing: final level & 0.3783 & 1.2e-10 & 6.5e-11 & 0.0211 \\
Sentencing: base level & 0.4786 & 3.8e-10 & 2.0e-10 & 0.0214 \\
Faculty: with rank & 0.0880 & 8.7e-11 & 9.0e-11 & 0.0089 \\
Faculty: without rank & 0.0882 & 7.5e-11 & 7.9e-11 & 0.0089 \\
Synthetic: WM penalty & 0.1046 & 1.2e-10 & 1.2e-10 & 0.0006 \\
Synthetic: reversed offsets & 0.1064 & 1.0e-10 & 9.9e-11 & 0.0010 \\
\bottomrule\end{tabular}\label{tab:bias-residual}\end{table}
The reported full-SAA reference KKT residuals are below $10^{-12}$. The largest reported Ghost equality residual is $2\cdot10^{-10}$ and inequality violations are zero. The larger stationarity residuals in Table~\ref{tab:bias-residual} preclude claiming a $10^{-3}$ KKT certificate for the empirical runs. Feasibility and optimization accuracy are distinct.
The White-male-penalty simulation has unconstrained known-merit MSE $1.561\cdot10^{-3}$ and mean Ghost known-merit MSE $1.445\cdot10^{-5}$. With reversed offsets these are $1.705\cdot10^{-3}$ and $4.537\cdot10^{-5}$. This demonstrates recovery of the synthetic rubric, not measurement of real-world hiring merit.
\begin{figure}[htbp]\centering\includegraphics[width=.98\linewidth]{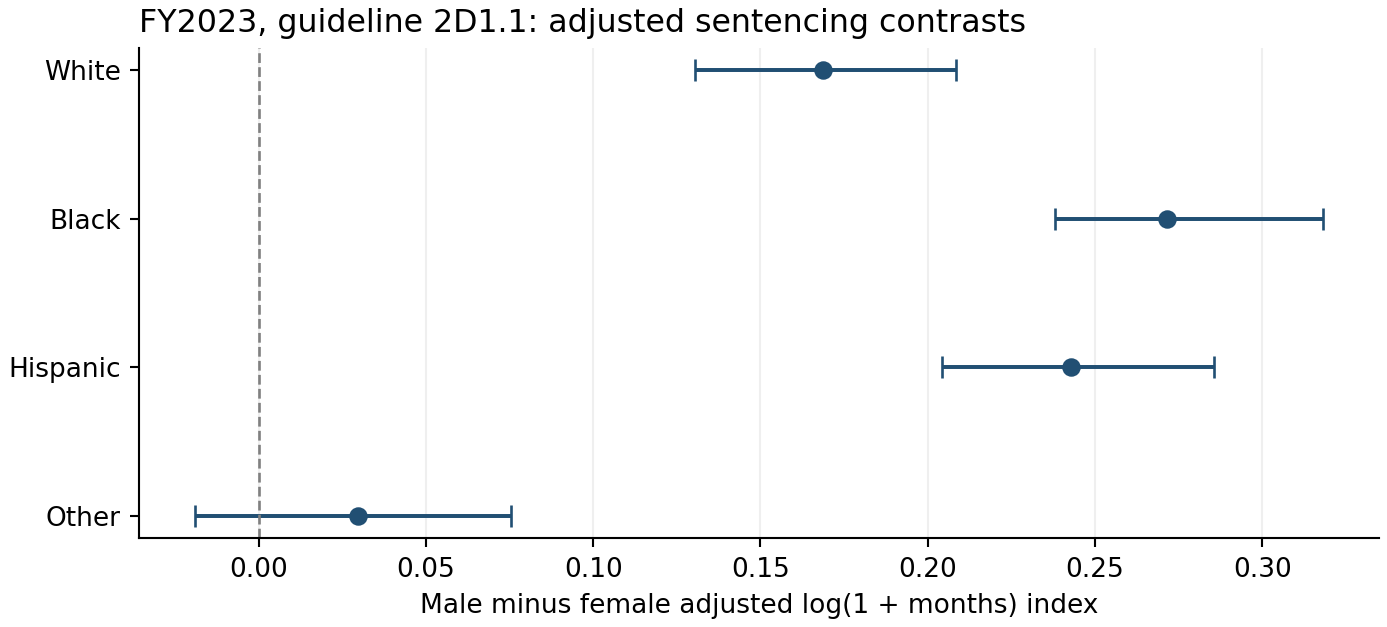}
\caption{Whole-cohort adjusted male-minus-female log-month index contrasts within each recorded race category. Bars are individual 95\% percentile intervals from 200 case-bootstrap refits. These exploratory associations differ from both the Commission report and the held-out optimization comparison.}\label{fig:sentencing_contrasts}\end{figure}
\begin{figure}[htbp]\centering\includegraphics[width=.98\linewidth]{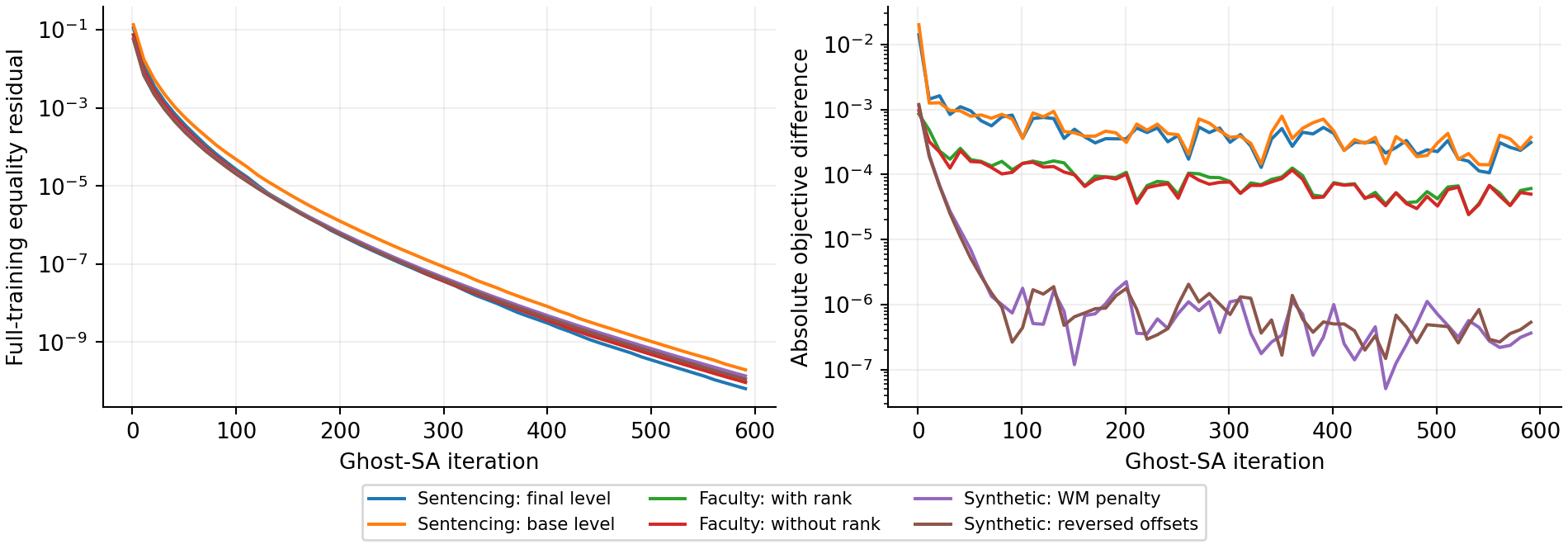}
\caption{Mixed equality residuals and absolute objective differences from the independent full-SAA solution, using median objectives over three seeds. Early infeasible iterates can have objectives below the feasible optimum.}\label{fig:bias_convergence}\end{figure}

The exploratory whole-cohort adjusted contrasts use 200
percentile bootstrap refits: resampling cases for sentencing and PhD institutions
for hiring, holding the feature transformations and reference profiles fixed. Optimizer-seed ranges measure numerical variability.

\section{Conclusions}
The ghost penalty provides a convergence analysis for stochastic directions
without adding a penalty parameter to the algorithm. Keeping constraint noise
signed until it has been accumulated yields an interval descent inequality
that is sufficient for asymptotic feasibility. The weak Sard condition then
allows convergence of penalty values to be combined with a finite-time
excursion argument, giving a vanishing population direction and stationarity
of every accumulation point.

The direction-oracle hypothesis is separate from this convergence analysis.
Uncapped multilevel randomization provides an unbiased direction with finite
variance and finite expected sampling cost when its antithetic corrections
have a second-moment decay exponent greater than one. The sufficient
condition proved here concerns the complete step map, including the sampled
relaxation. The feasibility-boundary example identifies a concrete obstruction
to deriving that rate solely from regularity of the outer quadratic program.
Practical fixed-batch directions, capped levels, inexact solves, and mixed
equality constraints require their own error or convergence analysis.

\section*{Code and results availability}
Implementation access is subject to written approval by
Vyacheslav Kungurtsev and to the applicable author-approval terms and
responsible-use policy. Requests, including the intended research purpose,
may be sent to \texttt{slavakung2@gmail.com}.

\section*{Acknowledgments}
Harsh Choudhary and Vyacheslav Kungurtsev acknowledge support from the Czech
National Science Foundation under Project 24-11664S.

\section*{Statement on use}
Earlier versions of this work referred to in the literature as ``StochasticGhost'' were associated with a project that turned out to be devoted to automating Marxist-Leninist ideology, and thus found to fall far short of basic acceptable ethical standards. The present study has been revised and is presented here solely for legitimate purposes. The authors explicitly disavow any unethical or harmful use of the methods described herein.

\section*{Generative AI}
ChatGPT was used to assist with both the theoretical validation and completion of the main results as well as refining and improving the numerical experiments.

\bibliographystyle{plain}
\bibliography{references}
\end{document}